\documentclass[11pt]{amsart}
\usepackage[english]{babel}
\usepackage{amssymb,enumerate,amsmath,amscd,txfonts,tikz}
 \usepackage{graphicx}
\usepackage[colorlinks=true,linkcolor=blue,citecolor=blue,urlcolor=blue]{hyperref}
\usepackage{tikz}
\usetikzlibrary{arrows,shapes,trees,backgrounds}

\newtheorem{theorem}{Theorem}[section]
\newtheorem{proposition}[theorem]{Proposition}
\newtheorem{lemma}[theorem]{Lemma}
\newtheorem{corollary}[theorem]{Corollary}

\theoremstyle{remark}

\newtheorem{remark}[theorem]{Remark}
\newtheorem{example}[theorem]{Example}
\numberwithin{equation}{theorem}

\renewcommand{\(}{\left(}
\renewcommand{\)}{\right)}

\newcommand{\Ric}{\operatorname{Ric}}

\begin{document}
\title[Gromov-Hausdorff precompactness of domains]{Precompactness of domains with lower Ricci curvature bound under Gromov-Hausdorff topology}
\author{Shicheng Xu}

\address{School of Mathematical Sciences, Capital Normal University, Beijing, China \newline
\indent Academy for Multidisciplinary Studies, Capital Normal University, Beijing, China}
\email{\href{mailto:shichengxu@gmail.com}{shichengxu@gmail.com}}

\date{\today}
\begin{abstract}
Based on a quantitative version of the classical Hopf-Rinow theorem in terms of the doubling property, we prove new precompactness principles in the (pointed) Gromov-Hausdorff topology for domains in (maybe incomplete) Riemannian manifolds with a lower Ricci curvature bound, 
which are applicable to those with weak regularities considered in PDE theory, and the covering spaces of balls naturally appear in the study of local geometry and topology of manifolds with lower curvature bounds.
All the new principles are more general than those earlier known for manifolds with smooth boundary, and improves those for manifolds with non-smooth boundary. 
\end{abstract}
{\maketitle}

\section{Introduction}\label{sec:intro}

The Gromov-Hausdorff topology has been one of the fundamental tools in the study of Riemannian manifolds under various curvature bounds. At the core is Gromov's precompactness principle \cite[Proposition 3.5]{GLP1981} (see also \cite[7.4.2 Compactness Thereom]{Burago-Burago-Ivanov2001}), which implies that the family of all complete Riemannian manifolds with a uniform lower Ricci curvature bound is precompact in the pointed Gromov-Hausdorff topology. The study of their limit spaces have been important and fruitful fields in the last 30 years, including Cheeger-Fukaya-Gromov's collapsing theory under bounded sectional curvature (\cite{Fukaya1988},\cite{Cheeger-Gromov1986},\cite{Cheeger-Gromov1990},\cite{Cheeger-Fukaya-Gromv}), Cheeger-Colding theory under lower bounded Ricci curvature (\cite{Cheeger-Colding1996},\cite{Cheeger-Colding1997}), and their synthetic generalizations (\cite{Burago-Gromov-Perelman1992},\cite{St1,St2},\cite{LV},\cite{AGS}).

In this paper we provide natural precompactness principles for domains in a Riemannian manifold with a lower Ricci curvature bound, which are suitable for those with weak regularities (\cite{Adams-Fournier1977}, \cite{Reifenberg1960}) considered in PDE theory, and are more general than the earlier known precompactness for manifolds with smooth boundary (\cite{Yamaguchi-Zhang2019},\cite{Knox2012},\cite{Wong2008},\cite{Anderson-Katsuda-Kurylev-Lassas-Taylor2004},\cite{Kodani1990}), and also substantially improves those for manifolds without boundary regularity conditions (\cite{Perales-Sormani2014}, \cite{Perales2020}, cf. a survey \cite{Perales2016}).

\subsection{Precompactness of domains whose boundary admits a uniform width}
Before proceeding into the main statements, let us fix some terminology first. Throughout the paper all (incomplete) Riemannian manifolds are assumed to have no boundary, i.e. every sufficient small open ball at a point is diffeomorphic to an Euclidean space. A domain $W$ in a Riemannian manifold $M$ is defined to be a connected subset whose topological interior $W^\circ$ and closure in $M$ satisfy $\overline{W^\circ}=\overline{W}$, where $W^\circ$ is not required to be connected. The natural length metric on $W$ induced by its metric tensor is denoted by $d_W$. The \emph{intrinsic boundary} $\partial W$ of $W$ is defined to be $W^\circ$'s topological boundary  in the completion $\overline{(W, d_W)}$. We define the $r$-interior $W_r^\circ$ to be $W_r^\circ=\{x\in W: d_W(x, \partial W)>r\}$. For a subset $S\subset W$, the open $r$-neighborhood of $S$ in $(W,d_W)$ are denoted by  $B_r(S,W)=\{x\in W: d_W(x,S)<r\}$.

\begin{theorem}\label{mthm-1}
	Let $r_0,t_0>0$ and $s(t)$ be a positive function on $t$ such that $s(t)\to 0$ as $t\to 0$. Let $\mathcal{D}(n, r_0, t_0, s(t))=\{(W, d_W, p)\}$ be the family of all domains $W$ in (maybe incomplete) Riemannian $n$-manifolds $M$, which are equipped with their length metric $d_W$ and base point $p$, such that 
	\begin{enumerate}\numberwithin{enumi}{theorem}
		\item\label{mthm-1.1} the $r_0$-neighborhood $B_{r_0}(W,M)$ of $W$ has a complete closure in $M$, and has Ricci curvature $\operatorname{Ric}_{B_{r_0}(W,M)}\ge - (n-1)$, and
		\item\label{mthm-1.3} for any $0<t\le t_0$, the closure of $s(t)$-neighborhood of $W_t^\circ$ in $(W,d_W)$ contains the whole $W$, i.e.,  $\overline{ B_{s(t)}(W_t^\circ,W)}\supset W$. (In particular, $W_t$ is nonempty)
	\end{enumerate}
	Then $\mathcal{D}(n, r_0, t_0, s)$ is precompact in the pointed Gromov-Hausdorff topology.
\end{theorem}

\begin{remark}
	All the conditions \eqref{mthm-1.1}-\eqref{mthm-1.3} are necessary in Theorem \ref{mthm-1}. 
	
	Indeed, \eqref{mthm-1.1} says that $W$ is $r_0$-definite away from the incomplete ``boundary'' of $M$, which avoids the possiblity of local concentration arising from the incompleteness of $M$. A counterexample that violates \eqref{mthm-1.1} can be easily found on the universal cover of a sphere with two antipodal points removed. For more details, see Example \ref{ex-petersen} below.

	A domain only satisfying \eqref{mthm-1.1}, in general, may have a too highly twisted shape around its boundary to admit the Gromov-Hausdorff precompactness; see Example \ref{ex-distortion} below. Condition \eqref{mthm-1.3} says that $W$ can be uniformly approximated from its interior, which will imply a uniform Gromov-Hausdorff precompactness locally around any boundary point of $W$; see Lemma \ref{lem-boundary} in Sec. \ref{sec-GH-precom}.
	
	Hereafter we call a domain $W$ satisfying \eqref{mthm-1.3} has a \emph{$s(t)$-undistorted boundary with width $\ge$ $t_0$}, or a $(s(t),t_0)$-undistorted domain for briefly.

\end{remark}
\begin{remark}
	By Lemma \ref{lem-balls} and Remark \ref{rem-condition1} below, the condition
	\eqref{mthm-1.1} can be replaced by a uniform Gromov-Hausdorff precompactness of all $r_0$-balls $B_{r_0}(x,M)$ at $x\in W$, i.e.
	\begin{enumerate}\numberwithin{enumi}{theorem}
		\item \label{mthm-1.4}
		$\{B_{r_0}(x,M): \forall x\in W, \text{ and } \forall W \in \mathcal D(r_0,t_0,s(t))\}$ is precompact in the Gromov-Hausdorff topology.
	\end{enumerate}

	The conditions \eqref{mthm-1.1} and \eqref{mthm-1.3} can be also replaced by other natural conditions suitable for manifolds with boundary without any regularity condition; see Theorem \ref{thm-int} below.
\end{remark}
\begin{remark}
	By the same proofs, Theorem \ref{mthm-1} and Theorem \ref{thm-lipschitz-type} below hold for domains in any length space $(X,d,\mu)$ with a positive Borel measure $\mu$ satisfying a uniform local doubling property, e.g., a $\operatorname{CD}(K,N)$ or $\operatorname{RCD}^*(K,N)$-space (\cite{St1,St2}, \cite{LV}, \cite{AGS}).
\end{remark}

\begin{remark}
	The (pointed) Gromov-Hausdorff distance always makes sense for metric spaces in which any bounded subset has a compact completion. Hence by our definition of domains, 
	both components of $W$'s interior $(W_i^\circ, d_{W_i^\circ})$ and $W$'s closure $(\overline{W_i}, d_{\overline{W_i}})$ with respect to their length metric are contained in Theorem \ref{mthm-1}. Moreover, Theorem \ref{mthm-1} also holds for the completion of $(W, d_W)$ with its length metric. Note that in general, the completion of $(W_i^\circ,d_{W_i^\circ})$ may be quite different from the closure $(\overline{W_i},d_{\overline{W_i}})$ in $M$.
\end{remark}

Let us recall from the standard PDE theory that, if $W$ is a domain with a uniform Lipschitz boundary in an Euclidean space $\mathbb R^n$, then there is $\tau>0$ such that $\overline{B_{\tau t}(W_t^\circ,\mathbb R^n)}\supset W$ (e.g.,  \cite[6.7 p143]{Gilbarg-Trudinger2001}), where $\tau$ depends the Lipschitz constant of the function by which $\partial \Omega$ is locally represented as its graph.

Similarly, if $s(t)$ can be chosen to be $\tau t$, then we call a (maybe unbounded) $(s(t),t_0)$-undistorted domain is of \emph{$\tau$-Lipschtiz type}. By defintion, $\tau\ge 1$. In the next theorem we give a uniform estimate on the global doubling propery for such domains. For simplicity we denote an open $r$-ball $B_r(p)$ when the ambient space is clear. 

\begin{theorem}\label{thm-lipschitz-type}
	Given an $\tau\ge 1$, $r_0,t_0>0$, there is a positive funtion $A_{n,r_0,t_0,\tau}:(0,+\infty)\to [1,\infty)$ monotone nondecreasing in $r$, such that any $(\tau t, t_0)$-undistorted domain $W$ of  $\tau$-Lipschitz type in $\mathcal{D}(n,r_0,t_0,\tau t)$, equipped with its length metric $d_W$ and Riemannian volume $\mu=\operatorname{vol}_g$, is globally $A_{n,r_0,t_0,\tau}(r)$-doubling, i.e.,
	for any $r>0$ and $x\in W$, 
	\begin{equation}
		0<\mu(B_r(x)) \le A_{n,r_0,t_0,\tau}(r)\cdot \mu(B_{\frac{r}{2}}(x))<\infty.
	\end{equation}
\end{theorem}

It is easy to check that domains satisfying the Reifenberg-flat \cite{Reifenberg1960} or cone condition \cite{Adams-Fournier1977} are of Lipschitz type defined here, and hence satisfy a global doubling property; see Proposition \ref{prop-examples} below. More generally, sub-level sets $W$ of a distance function in a complete Riemannian manifold $M$, e.g. a $r$-neighborhood of a connected subset $K$, are $(t, r)$-undistorted domains of $1$-Lipschitz type. 

An application of Theorem \ref{mthm-1} is an improvement of the almost rigidity of warped product by Cheeger-Colding \cite{Cheeger-Colding1996} under the hypothesis of \cite[Theorem 3.6]{Cheeger-Colding1996}; see Corollary \ref{cor-1} and Theorem \ref{thm-almost-rigidity} in Sec. \ref{sec:precompact-non-length}.

\subsection{Precompactness of covering spaces of metric balls}
A partial motivation of Theorem \ref{mthm-1} is the lack of precompactness of the local universal covers of Riemannian manifolds.
The universal cover $\pi:(\widetilde{B_r(p)},\tilde p)\to (B_r(p),p)$ of a ball has taken an important role in the study of local geometry and topology of spaces with a lower Ricci or sectional curvature bound, such as the Margulis lemma (\cite{Fukaya-Yamaguchi1992},\cite{Cheeger-Colding1996},\cite{Kapovitch-Wilking2011},\cite{Xu-Yao}), collapsing of manifolds with local bounded covering geometry (\cite{Chen-Rong-Xu2016,Chen-Rong-Xu2018},\cite{Huang-Kong-Rong-Xu2020}), the existence of universal cover of Ricci limit spaces by \cite{Sormani-Wei2001,Sormani-Wei2004}, and the semilocal simply-connectedness of a Ricci-limit space proved recently by \cite{Pan-Wei2019} and \cite{Wang2021}.
However, as pointed out by Sormani-Wei \cite{Sormani-Wei2004}, a sequence of Riemannian universal covers $\widetilde{B_r(p_i)}$ generally admit no converging subsequence in the pointed Gromov-Hausdorff topology, even $B_r(p_i)$ lies in manifolds with a uniform lower sectional curvature bound; see Examples \ref{ex-sw} and \ref{ex-petersen} below.

A corollary of the main theorems is a new precompactness of normal covers of balls which are readily applied in all the settings above.

\begin{corollary}[Precompactness of normal covers]\label{cor-pre}
	 For any $0<r_1<r_2$,
	 let $\widehat{B}(p,r_1,r_2)$ be a component of the preimage $\pi^{-1}(B_{r_1}(p))$ in the Riemannian universal cover $\pi:(\widetilde{B_{r_2}(p)},\tilde p)\to (B_{r_2}(p),p)$, where $B_{r_2}(x)$ is relative compact in a Riemannian $n$-manifold and has $\operatorname{Ric}\ge -(n-1)$.  Let $\imath:\pi_1(B_{r_1}(p),p)\to \pi_1(B_{r_2}(p),p)$ be the homomorphism between fundamental groups induced by the inclusion map.
	\begin{enumerate}
		\item \label{cor-pre-i1}
		$\widehat{B}(p,r_1,r_2)$ equipped with its length metric and Riemannian volume is globally $A_{n,r_1,r_2}(r)$-doubling, where $A_{n,r_1,r_2}(r)$ is a positive function non-decreasing in $r$ when $n$, $r_1$ and $r_2$ are fixed.
		\item \label{cor-pre-i2}  The family consisting of all normal covering spaces of $B_{r_1}(p)$, $$\(\widetilde{B_{r_1}(p)}/\operatorname{Ker}(\imath), \hat p,  \Gamma(r_1,r_2)\)=(\widehat{B}(p,r_1,r_2),\hat p, \operatorname{Im}(\imath))\to (B_{r_1}(p),p),$$ together with deck transformations by $\Gamma(r_1,r_2)=\pi_1(B_{r_1}(p),p)/\operatorname{Ker}(\imath)$,
		is precompact in the pointed equivariant Gromov-Hausdorff topology.
	\end{enumerate}
\end{corollary}
\begin{remark}
	Corollary \ref{cor-pre} directly implies and also extends the precompactness by Sormani-Wei \cite[Proposition 3.1]{Sormani-Wei2004} of relative $\delta$-covers of balls, where a relative $\delta$-cover $\widetilde B(p, r_1, r_2)^\delta$ ($0<r_1<r_2$, $\delta>0$) is a connected lift of $B_{r_1}(p)$ in the $\delta$-cover of $B_{r_2}(p)$.
	This is because $\widehat{B}(p,r_1,r_2)$ naturally covers $\widetilde B(p, r_1, r_2)^\delta$.
 \end{remark}
Because $B_{r_1}(p)$ lies in $\mathcal{D}(n,\frac{r_2-r_1}{2},r_1,t)$ is a $1$-Lipschitz typed domain with width $\ge r_1$, by the lifting property so is $\pi^{-1}(B_{r_1}(p))$ in $\widetilde{B_{r_2}(p)}$. Now \eqref{cor-pre-i1} follows directly from Theorem \ref{thm-lipschitz-type}.

For \eqref{cor-pre-i2}, note that, for a loop $\gamma\in \pi_1(B_{r_2}(p),p)$, $\gamma(\tilde p)\in \widetilde{B_{r_2}(p)}$ lies in the component $\widehat{B}(p,r_1,r_2)$ of $\pi^{-1}\(B_{r_1}(p)\)$ containing $\tilde p$, if and only if $\gamma$ can be represented by a loop in $B_{r_1}(p)$. Hence $\operatorname{Im}(\imath)$ acts isometrically on $\widehat{B}(p,r_1,r_2)$.
On the other hand, $\widehat{B}(p,r_1,r_2)$ is the covering space of $B_{r_1}(p)$ such that any loop $\gamma$ in $\pi_1(B_{r_1}(p),p)$  homotopy trivial in $B_{r_2}(p)$ still lifts to a loop in $\widehat{B}(p,r_1,r_2)$. Hence, $\widehat{B}(p,r_1,r_2)$ coincides with
$\widetilde{B_{r_1}(p)}/\operatorname{Ker}(\imath)$, which is a normal cover of $B_{r_1}(p)$. By \cite[Proposition 3.6]{Fukaya-Yamaguchi1992} and Theorem \ref{mthm-1}, \eqref{cor-pre-i2} holds.

Similarly, all covering spaces $(\widehat{W},\hat p)$ of a domain $W=\cap_{i} B_r(p_i)$ which is the intersection of $r$-balls in complete Riemannian $n$-manifolds with $\operatorname{Ric}\ge -(n-1)$ is also pointed-Gromov-Hausdorff precompact, as long as the fundamental group $\pi_1(\widehat{W},\hat p)$ contains all loops that are homotopy trivial in a definite larger neighborhood of $W$, e.g., $B_{r_0}(W)$ or $W_{r+r_0}=\cap_i B_{r+r_0}(p_i)$. 

As an application of Corollary \ref{cor-pre}, it can be used to simplify the proof of a key lemma \cite[Lemma 3.1]{Wang2021} in the proof of semi-local simply-connectedness of a Ricci-limit space, where the original proof by Wang \cite{Wang2021} involves the slice theorem for pseudo-group actions proved in Pan-Wang \cite{Pan-Wang2021} due to the lack of a precompactness result for covering spaces of metric balls; see Sec. \ref{sec:semilocal-sim}. For further applications see \cite{Jiang-Kong-Xu}. 

\subsection{Precompactness of manifolds with (not necessary smooth) boundary}
Let us consider the case of manifolds with boundary, where \eqref{mthm-1.1} is absent. The earlier known precompactness in Kodani \cite{Kodani1990}, Anderson-Katsuda-Kurylev-Lassas-Taylor \cite{Anderson-Katsuda-Kurylev-Lassas-Taylor2004} and Knox \cite{Knox2012} are for manifolds with smooth boundary, whose  interiors are assumed to be non-collapsed under bounded sectional or Ricci curvature, and whose boundaries admit uniform regularity conditions (in order to derive $C^{1,\alpha}$-convergence). For example, the boundary injectivity radius in \cite{Anderson-Katsuda-Kurylev-Lassas-Taylor2004}, \cite{Kodani1990} and also \cite{Knox2012} satisfies, $i_{\partial M}\ge i_0$, where $i_{\partial M}$ is defined to be the maximal $t>0$ such that the normal exponential map from $T^{\perp}_t\partial M=\{v\in T^\perp\partial M: |v|<t\}$ is injective.

Perales-Sormani \cite{Perales-Sormani2014} and Perales \cite{Perales2020} further extended the precompactness to manifolds $M$ with arbitrary boundary.  As before, any (incomplete) Riemannian manifold is assumed to contain only interior points, i.e., every sufficient small open ball at a point is diffeomorphic to an Euclidean space. And let its $r$-interior $M_r^\circ$ (called a $r$-inner region in \cite{Perales-Sormani2014}) and boundary $\partial M$ be defined by the same way, i.e., $M_r^\circ=\{x\in M:d_M(x,\partial M)>r\}$, and $\partial M$ is the boundary of $M$ in its completion $\overline M$. The length metric on $M_r^\circ$ is denoted as usual by $d_{M_r^\circ}$, where $d_{M_r^\circ}(x,y)=\infty$ whenever the two points lie in different components of $M_r^\circ$.

\begin{theorem}[\cite{Perales-Sormani2014}, \cite{Perales2020}]\label{thm-PS}
	The following families are precompact in the Gromov-Hausdorff topology, provided that each of $\{(M,d_M)\}$ below is a (maybe incomplete) Riemannian $n$-manifold such that
	$\Ric_M\ge0$, the volume $\operatorname{vol}(M)\le V$, and $\operatorname{vol}(B_r(q,M))\ge \theta r^n$ for some $q\in M_r^\circ$.
	
	\begin{enumerate}
		\item\label{thm-PS-1} The family consisting of $r$-interiors $(M_r^\circ, d_M)$, equipped with  restricted metric $d_M$, such that their intrinsic diameter $\operatorname{diam}(M_r^\circ, d_{M_r^\circ})\le D$.
		
		\item\label{thm-PS-2} The family $\{(M,d_M)\}$ of manifolds above such that $\operatorname{diam}(M_r^\circ, d_{M_r^\circ})\le D$ and $i_{\partial M}\ge r$.
		
		\item\label{thm-PS-3} The family $\{(M,d_M)\}$ of manifolds above such that for some $0<r_i\to 0$,  $\operatorname{diam}(M_{r_i}^\circ,d_{M_{r_i}^\circ})\le D$, and their boundaries $\{(\partial M,d_M)\}$ equipped with restricted metric form a precompact family in the Gromov-Hausdorff topology.
	\end{enumerate}
\end{theorem}

Our next result is an improvement of Theorem \ref{thm-PS} for the general case.

\begin{theorem}\label{thm-int}
	Let $n>0$ be a integer and $r,D>0$ be real numbers. The following families are precompact in the Gromov-Hausdorff topology. 
	\begin{enumerate}
		\item\label{thm-int-1}  The family $\mathcal{M}_\imath(r,D,n)=\{(M_r^\circ, d_M): \operatorname{diam}(M_r^\circ, d_{M_r^\circ})\le D,\; \dim M=n,\; \Ric_M\ge-(n-1)\}$. 
		\item\label{thm-int-2} 
		The family $\mathcal{M}_{\partial}(r,D,n)$  consisting of $n$-manifolds whose $\Ric_M\ge -(n-1)$, and $r$-interior $M_r^\circ$ contains a connected subset $W$ such that $\operatorname{diam}(W,d_{W})\le D$ and  $\overline{B_{r}(W,M)}=M$.
		\item\label{thm-int-3} The family  $\mathcal{M}_b(r_i,D,n)$ consisting of $n$-manifolds whose $\Ric_M\ge -(n-1)$ such that, for $0<r_i\to 0$,  $\operatorname{diam}(M_{r_i}^\circ,d_{M_{r_i}^\circ})\le D$ (in particular, $M_{r_i}^\circ$ is either connected or  empty), and their boundaries $\partial \mathcal{M}_b(r,D,n)=\{(\partial M,d_M)\}$ form a precompact family in the Gromov-Hausdorff topology.
		\end{enumerate} 
	\end{theorem}
	In the above, $r$-interiors $M_r^\circ$ in \eqref{thm-int-1} can be replaced by their components whose intrinsic diameter $\le D$. Similarly, the precompactness in \eqref{thm-int-2} and \eqref{thm-int-3} holds for the case that $M_r^\circ$ has at most $N$ components, each of which has intrinsic diameter $\le D$.

\begin{remark}
	All the diameter conditions are necessarily required in Theorem \ref{thm-int}. In Example \ref{ex-petersen} below a $r$-interior $(M_r^\circ,d_{M_r^\circ})$ has infinite intrinsic diameter, but bounded extrinsic diameter $\le \pi$.
	
	For the case that $\operatorname{diam}(M_r^\circ,d_{M_r^\circ})$ has no upper bound, we also have a Gromov-Hausdorff precompactness for intrinsic balls in $M_r^\circ$, as a counterpart of \eqref{thm-int-1}.
	\begin{enumerate}
		\item \label{thm-int-4} The family  $\mathcal{M}(r,R,n)=\{(B_R(p,M_r^\circ), p, d_M): \dim M=n, \Ric_M\ge -(n-1)\}$, which consists of all $R$-balls $B_R(p,M_r^\circ)$ in $(M_r^\circ,d_{M_r^\circ})$, but equipped with the restricted metric $d_M$ of Riemannian $n$-manifolds $M$ with $\Ric_M\ge -(n-1)$, is precompact in the Gromov-Hausdorff topology. 
	\end{enumerate}
\end{remark}

\begin{remark}
	Compared with Theorem \ref{thm-PS} and \cite{Knox2012},  \cite{Anderson-Katsuda-Kurylev-Lassas-Taylor2004}, \cite{Kodani1990}, there is no non-collapsing condition for $M$ in Theorem \ref{thm-int}.

	The condition in \eqref{thm-int-2} is weaker than that in \eqref{thm-PS-2}, because $i_{\partial}(M)\ge i_0$ implies $\overline{B_{r}(M_r^\circ,M)}=M$ for any $0<r\le i_0$. The case that $W\neq M_r^\circ$ naturally appears in \eqref{thm-int-2} can be found in Corollary \ref{cor-1} below.
	
	\eqref{thm-int-2} includes those manifolds considered in Knox \cite{Knox2012}, Anderson-Katsuda-Kurylev-Lassas-Taylor \cite{Anderson-Katsuda-Kurylev-Lassas-Taylor2004} and Kodani \cite{Kodani1990}, which also require $i_{\partial}(M)\ge i_0$.

\end{remark}

\subsection{Precomapctness for manifolds with no positively lower bounded width}

It is worth noting that, \eqref{thm-int-3} includes inradius collapsed manifolds considered by Yamaguchi-Zhang \cite{Yamaguchi-Zhang2019}, i.e. for any $r>0$, there is $M$ such that $M_r^\circ$ is empty. In comparison, all domains in Theorem \ref{mthm-1} admit a non-empty $t_0$-interior.

Inradius collapsed manifolds, whose sectional curvature $\ge -1$ or Ricci curvature $\ge -(n-1)$ and the second fundamental form of whose boundary $|II_{\partial M}|\le \lambda$, earlier are known  to admit Gromov-Hausdorff precompactness by Wong \cite{Wong2008}, where the proof provides examples satisfying \eqref{mthm-1.1} (but not necessarily satisfy \eqref{mthm-1.3}). 
\begin{example}\label{ex-wong}
	Based on Kosovskii's work \cite{Kosovskii} (cf. \cite{Perelman1997}), it was proved by Wong \cite{Wong2008} (cf. Yamaguchi-Zhang \cite{Yamaguchi-Zhang2019})
	that given any $r_0>0$ and $\epsilon>0$, a complete manifold with boundary, whose sectional curvature $\ge -1$ and the second fundamental form of the boundary $|II_{\partial M}|\le \lambda$, admits an extension by gluing a warped cylinder $\partial M\times_{\phi}[0,r_0]$ along $\partial M$ such that $\phi(0)=1$,  $\phi(r_0)=\epsilon$, and the resulting
	manifold $M'$ is an Alexandrov space with curvature $\ge K(\lambda,r_0,\epsilon)$, and the boundary $\partial M'$ is totally geodesic. 	
	Hence a corresponding version of \eqref{mthm-1.1} holds for all such manifolds $M$ in their extensions $M'$. 
	
	By the warped structure of $\partial M\times_{\phi}[0,r_0]$, $(M')_t^\circ=M\cup_{\partial M} \partial M\times_\phi [0,r_0-t)$ satisfies \eqref{mthm-1.1} and \eqref{mthm-1.3}. Hence by Theorem \ref{mthm-1}, the family $\{(M')_t^\circ\}$ consisting of all such manifolds is pointed Gromov-Hausdorff precompact. Moreover, since the metrics $d_M$ and $d_{(M')_t^\circ}$ are $\epsilon$-Lipschitz equivalent to each other, so are the original manifolds $M$.
	
	Therefore, inradius collapsed manifolds considered in \cite{Wong2008} and \cite{Yamaguchi-Zhang2019} essentially are also covered by Theorem \ref{mthm-1}.

\end{example}

In Example \ref{ex-wong}, the boundary set $\{(\partial M, d_M)\}$ is precompact in the pointed Gromov-Hausdorff topology. 
Conversely, it also provides Gromov-Hausdorff precompactness for manifolds with non-smooth boundary and with a lower curvature bounded extension, which is a natural extension of Wong \cite{Wong2008} and Yamaguchi-Zhang \cite{Yamaguchi-Zhang2019}.
 
\begin{theorem}\label{thm-cptbdy}
	Let $\mathcal{D}_c(n, r_0)=\{(W,d_W,p\in \partial W)\}$ be a family of closed domains in Riemannian $n$-manifolds $M$ such that \eqref{mthm-1.1} holds for $r_0>0$, and the set of boundaries, $\partial \mathcal{D}_c(n, r_0)=\{(\partial W,d_W,p):W\in \mathcal{D}_c(n, r_0)\}$, is a precompact family in the pointed Gromov-Hausdorff topology. Then $\mathcal{D}_c(n,r_0)$ itself is precompact in the pointed Gromov-Hausdorff topology. 
\end{theorem}
\begin{remark}
	By Theorem \ref{thm-int}, the condition \eqref{mthm-1.1} in Theorem \ref{thm-cptbdy} can also replaced with the extrinsic diameter bound of $M_{r_i}^\circ$ in \eqref{thm-int-3},
	either of which is necessarily required for the precompactness; see Example \ref{ex-petersen} for counterexamples.
\end{remark}

The rest of the paper is organized as follows: In Sec. \ref{sec:exam}, we give examples and counterexamples of the main theorems. 
In Sec. \ref{sec:quan-HR} we prove Theorem \ref{thm-lipschitz-type}. In Sec. \ref{sec-GH-precom} we prove Theorems \ref{mthm-1}, \ref{thm-int} and \ref{thm-cptbdy}. In Sec. \ref{sec:precompact-non-length} Gromov-Hausdorff precompactness for connected subsets with some non-length metric are derived. In Sec. \ref{sec:semilocal-sim} we give an application of Corollary \ref{cor-pre}, which yields a simplified proof of a key lemma in proving the semi-local simply-connectedness of a Ricci-limit space. 

{\bf Acknowledgment.} Supported in part by NSFC 11821101 and 12271372. The author is grateful to Professor Xiaochun Rong for proposing the problem on the precompactness of covering spaces of metric balls, and Professor Christina Sormani for pointing out earlier works on the same topic and valuable comments that led to improvement of the comparison with earlier results.

\section{Examples}\label{sec:exam}

Let us first recall that, the Gromov-Hausdorff pseudo-distance $d_{GH}$ between metric spaces $(X_i, d_{X_i})$ ($i=1,2$) is the infimum of those $\epsilon>0$ such that there exists a metric $d$ on the disjoint union $X_1\sqcup X_2$ satisfying that $d|_{X_i}=d_{X_i}$ and $X_i$ is $\epsilon$-dense in $(X_1\sqcup X_2,d)$. It induces naturally the Gromov-Hausdorff topology on all metric spaces.

Gromov (\cite[Proposition 3.5]{GLP1981}, see also e.g. \cite[\S 11.1.4]{Petersen2016}, \cite[\S 7.4]{Burago-Burago-Ivanov2001}) observed the following \emph{precompactness principle}.

Let $\mathcal{M}$ be a family of metric spaces, each of which has a compact completion and diameter $\le D$. Then $\mathcal{M}$ is precompact in the Gromov-Hausdorff topology (i.e., any sequence in $\mathcal{M}$ admits a Cauchy subsequence in the Gromov-Hausdorff pseduo-distance), if and only if $\mathcal{M}$ is uniformly totally bounded. That is, there is $\epsilon_0>0$ and a positive function $C(\epsilon)$ such that for any $0<\epsilon\le \epsilon_0$, the $\epsilon$-covering number $\operatorname{Cov}_\epsilon(X,d)$ is uniformly bounded by $C(\epsilon)$, where $\operatorname{Cov}_\epsilon(X,d)$ is defined to be the least number of $\epsilon$-balls whose union covers $X$.

In the above  $\operatorname{Cov}_\epsilon(X,d)$ can be equivalently replaced by the $\epsilon$-packing number, $\operatorname{Cap}_\epsilon(X,d)$, which is defined to be the maximal count of points in $(X,d)$ whose pairwise distance $\ge \epsilon$ (for simplicity, $\epsilon$-discrete). This is because $\operatorname{Cov}_\epsilon(X,d)\le \operatorname{Cap}_\epsilon(X,d)\le \operatorname{Cov}_{\frac{\epsilon}{2}}(X,d)$.

For a sequence of metric spaces $(X_i, d_i, p_i)$ with a base point, it converges to $(X_\infty, d_\infty, p_\infty)$ in the \emph{pointed Gromov-Hausdorff topology} is defined to be that for all $0\le r<\infty$, $B_r(p_i,X_i)$ converges to $B_r(p_\infty,X_\infty)$.

Let $\mathcal{M}_p$ be a family of metric spaces $(X,d,p)$ with a base point $p\in X$, in which each ball $B_R(p)$ at $p$ has a compact completion. Then by Gromov's precompactness principle above, $\mathcal{M}_p$ is precompact in the pointed Gromov-Hausdorff topology, if and only if for any $R>0$, the family $\left\{\(B_R(p,X), d|_{B_R(p,X)}\)\right\}$ is uniformly totally bounded by $C_R(\epsilon)$ as $0<\epsilon\le \epsilon_0(R)$.

Together with Bishop-Gromov's relative volume comparison, the above principle implies Gromov's precompactness theorem, i.e., the family consisting of all complete manifolds with a lower bounded Ricci curvature is precompact.

Now we give some examples of domains in (incomplete) Riemannian manifolds whose sectional curvature is non-negative, but they admit no Cauchy subsequence in the pointed Gromov-Hausdorff topology.
 
\begin{example}[Sormani-Wei, \cite{Sormani-Wei2004}]\label{ex-sw}
	Let $(\mathbb RP^2, d_k)$ be the real projective plane with different metrics, where $d_k$ $(k\ge 2)$ is the induced length metric from a regular polygon $P_k$ with $2k$ sides such that the opposite sides are glued together in reversed direction, and the distance from polygon's center $o_k$ to each side equals to $1$. Then the distance from $o_k$ to each vertex is $\frac{1}{\cos \frac{\pi}{2k}}$. Let $r_k=\frac{1}{2}\left(1+\frac{1}{\cos \frac{\pi}{2k}}\right)$. Then the boundary $\partial B_{r_k}(o_k)$ of the open ball $B_{r_k}(o_k)$ has $k$ components. Let us take $U_k$ to be the open ball $B_{r_k}(o_k)$ or its closure. Then the universal cover $(\widetilde{U}_k,\tilde o_k)$ with the lifted metric $\tilde d_k$ and base point $\tilde o_k\in \pi_k^{-1}(o_k)$ admits no Gromov-Hausdorff convergent subsequence as $k\to \infty$. 
	
	This is because the metric ball $B_4(\tilde o_k, \widetilde U_k)$ in $(\widetilde{U}_k,\tilde d_k)$ contains at least $k$ preimages of $o_k$, whose pairwise distance is no less than $2$. Since $d_k$ can be smoothed around each vertex to a Riemannian metric with nonnegative sectional curvature, it gives a counterexample that a sequence of the universal cover of balls has no converging subsequence.
	
	Note that $\widetilde U_k$ are undistorted in $1$-Lipschitz type with width $\ge 1$. But $\widetilde U_k$ does not admit a complete closure as required in \eqref{mthm-1.1}.
\end{example}

In the above example, the $\epsilon$-packing number $\operatorname{Cap}(B_4(\tilde o_k,\widetilde U_k))$ fails to admit a uniform bound due to that the fundamental group $\pi_1(U_k)$ has more and more generators as $k\to \infty$. Next, we give an example whose base spaces have fundamental group $\mathbb Z$, which illustrates that the phenomena of non-precompactness for universal covers does not due to the topological complexity of the base space.

\begin{example}[cf. {\cite[Example 6.3.4]{Petersen2016}}] \label{ex-petersen}
	 Let $D^2$ be half of the standard northern hemisphere, whose boundary $\partial D^2$ consists of two antipodal points $p,p^*$ and two great circle arcs between $p$ and $p^*$. By gluing the arcs together with respect to their arc length with $p,p^*$ fixed, we get a metric space $(\mathbb S^2, d_0)$ homeomorphic to the $2$-sphere. Let us take a point $o$ lying in the middle between $p$ and $p^*$. Then $U=B_{\frac{\pi}{2}}(o)$ covers $(\mathbb S^2, d_0)$ except the two points $p,p^*$, and $\pi_1(U)=\mathbb Z$. Let us consider the universal cover $\pi:(\widetilde{U},\tilde o)\to (U, o)$. Note that for any $\epsilon>0$, any point $\tilde x$ in $\widetilde{U}$ can be joined to $\tilde o$ by a curve with length $\le \pi+\epsilon$, which consists of the lifting of two great circle arcs (one is from $\pi(\tilde x)$ to $p$ or $p^*$, another passes $o$), and a small circle around $p$ or $p^*$, where the circle's length can be chosen arbitrary small. It follows that the closure $\overline {B_{\pi}(\tilde o,\widetilde{U})}$ covers the whole incomplete manifold $\widetilde{U}$, and thus contains infinitely many of preimage points of $o$.
	
	For counterexamples, let $W_i$ be $\overline{B_{\frac{\pi}{2}-\epsilon_i}(o)}$ in $(\mathbb S^2, d_0)$ with $\epsilon_i\to 0$, and let us take its universal cover $\widetilde{W}_i$, which admits a complete closure in $M=\widetilde{U}$. Then $\widetilde{W}_i$ has an infinity diameter, but the number of preimage points of $o$ in $B_{2\pi}(\tilde o)\subset \widetilde{W}_i$ tends to infinity as $i\to \infty$.
	
	Similar to Example \ref{ex-sw}, $\widetilde W_i$ are $(t,1)$-undistorted domains of $1$-Lipschitz type. But for any $r>0$, the $r$-neighborhood of $\widetilde W_i$ in $M$ does not admit a complete closure for sufficient large $i$. Hence \eqref{mthm-1.1} fails for $\widetilde{W}_i$.
\end{example}

\begin{remark}\label{rem-ex-petersen}
	Intuitively, the failure of Gromov-Hausdorff precompactness on the universal covers $\widetilde{B_r(p_i)}$ of $r$-balls $B_r(p_i)$ in complete Riemannian manifolds with a lower Ricci curvature bound is due to that, when near its boundary, there is a non-trivial loop in $\pi_1(B_r(p_i))$ whose length tends to $0$, which leads to a high twisting around a point in $\partial\widetilde{B_r(p_i)}$ such that more and more points are concentrated together. Similar examples can be found in the flat torus $T^n$, where $W_i$ is the universal cover of balls whose radius approaches to $T^n$'s diameter.
	
	Corollary \ref{cor-pre} illustrates that it is indeed the case, by showing the precompactness of domains over smaller balls definite away from $\partial U_k$. 
	
	On the other hand, by Theorem \ref{thm-int} if the preimage of definite smaller balls (e.g. $B_{r/2}(p_i)$) in $\widetilde{B_r(p_i)}$ has components at most $N$, and each component has intrinsic diameter $\le D$, then $\widetilde{B_r(p_i)}$ themselves admit Gromov-Hausdorff precompactness.
\end{remark}

A counterexample without uniform undistortedness can be easily constructed as in below.

\begin{example}\label{ex-distortion}
	Let us choose countable and infinite rays $L_k$ starting from the coordinate origin $o$ in the Euclidean plane, whose directions $\xi_k=(1, \frac{1}{2^k})$. Let $U_i$ be the domain consisting of $\frac{1}{100^k}$-neighborhood of $L_k$ for $k\le i$, i.e., $U_i=\cup_{k=1}^iB_{ \frac{1}{100^k}}(L_k)$. Then \eqref{mthm-1.3} fails for $\{U_i\}$, and $(U_i,o)$ with its length metric admits no Gromov-Hausdorff converging subsequence. 
	
	Similarly, the open set on the Euclidean plane bounded by the topologist's sine curve $\{(x,\sin \frac{1}{x}): x>0\}$, the $y$-axis, and the horizontal ray $\{(x, -2): x>0\}$ is a domain whose boundary is not uniformly undistorted.
\end{example}

In the remaining of this section, we verify that domains in an Euclidean space whose boundary is Reifenberg-flat or satisfies the cone condition are undistorted in Lipschitz type.

Let us recall that for $\theta,H>0$, a domain $W$ satisfies the $(\theta, H)$-cone condition \cite{Adams-Fournier1977}, if any point $x\in W$ there is a finited spherical cone $C(x,\vec{v},\theta,H)=\{y : \angle(\overrightarrow{xy},\vec{v}(x))\le \theta, d(x,y)\le H\}$ with vertex $x$, angle $\theta$ and length $H$  contained in $W$.

An open domain $W$ satisfies the $(\epsilon,r_0)$-Reifenberg-flat condition with a local separability property \cite{Lemenant-Milakis-Spinolo}, if  (a)
for every $x\in \partial W$, and every $0<r\le r_0$, there is a hyperplane $P(x,r)$ containing
$x$ such that the Hausdorff distance between $\partial W\cap B_r(x)$ and $P(x,r)\cap B_r(x)$ is no more than $r\epsilon$;
(b)
for every $x\in \partial W$, one of the connected component of
$B(x, r_0)\cap\{y : d(y, P(x, r_0))\ge 2\epsilon r_0\}$
is contained in $W$ and the other one is contained in $\mathbb R^n\setminus W$.

\begin{proposition}\label{prop-examples}
	Let $W$ be an open domain in $\mathbb R^n$.
	\begin{enumerate}
		\item If $W$ satisfies $(\theta, H)$-cone condition, then $W$ is a $(t/\sin\theta,H)$-undistorted domain of $1//\sin\theta$-Lipschitz type. 
		\item If $W$ satisfies $(\epsilon,r_0)$-Reifenberg-flat condition ($\epsilon\le 1/600$) with local separability property, then $W$ is a $(810450t,r_0/12607)$-undistorted domain of $810450$-Lipschitz type.
	\end{enumerate}
\end{proposition}
\begin{proof}
	First let us assume that $W$ satisfies $(\theta,H)$-interior cone condition. Then for any $t\le H$, the $t$-interior $W_t^\circ$ contains the $t$-interior of $(\theta,H)$-cone for all $x\in W$. Hence, the closure of $\frac{t}{\sin\theta}$-neighborhood of $W_t^\circ$ contains the whole $W$.
	
	Secondly let us consider a $(\epsilon,r_0)$-Reifenberg-flat domain $W$. By Lemenant-Milakis-Spinolo's \cite[Theorem 3]{Lemenant-Milakis-Spinolo}, for $0<\epsilon\le 1/600$, $W$ is $(1/450, r_0/7)$-Jones flat. That is,
	for any $x,y\in W$ with $d(x,y)\le r_0/7$, there is a rectifiable curve $\gamma\subset W$ with length $\le 450d(x,y)$ satisfying for any $z\in \gamma$, $d(z, \partial W)\ge \frac{d(z, x)d(z, y)}{450 d(x, y)}$.
	It follows that for any $x,y\in W$ with distance $d(x,y)=r\le r_0/7$, there is a curve $\gamma$ and $z\in \gamma$ such that $d(x,z)=d(y,z)\ge \frac{r}{2}$, $d(z,\partial W)\ge \frac{d(x,z)^2}{450 r}\ge \frac{r}{1800}$. Hence $W$ is contained in the $450r$-neighborhood of $W_{r/1801}^\circ$.
\end{proof}

\section{Quantitative Hopf-Rinow Theorem via doubling Borel measures}\label{sec:quan-HR}
In this section we prove Theorem \ref{thm-lipschitz-type} for undistorted domains of Lipschitz type. 

As preliminaries, we first prove a quantitative version of the Hopf-Rinow theorem (\cite{Hopf-Rinow1931}, \cite{Cohn-Vossen1936}), whose settings and proofs are almost standard. It also provides basic ideas and technical tools applied later in the next section. 

Recall that a connected Riemannian manifold 
$(M, g)$ is called to be \emph{geodesically complete}, if every geodesic in $M$ is always extendable and thus its parameter is defined on the whole $\mathbb R^1$. For a connected Riemannian manifold, the Hopf-Rinow theorem \cite{Hopf-Rinow1931} says that $(M,g)$ is geodesically complete if and only if $(M,d)$ with its length metric $d$ is a complete metric space. Moreover, if $(M,d)$ is geodesically complete, then it is also \emph{proper}, i.e., every bounded and closed subset is compact (also called boundedly compact, or satisfies the Heine-Borel property). 

The Hopf-Rinow theorem has been generalized to length spaces. A metric space $(X,d)$ is called a \emph{length space}, and $d$ a length-metric, if for any $x,y\in X$, $d(x,y)$ is realized by the infimum of length of all continuous paths connecting $p$ and $q$.
The Hopf-Rinow-Cohn-Vossen theorem (\cite{Cohn-Vossen1936}, see also \cite[Theorem 2.5.28]{Burago-Burago-Ivanov2001}) says that, for any complete length metric space $(X,d)$, the local compactness of $(X,d)$ is equivalent to be proper, and thus either of them implies $X$ is geodesic (i.e., any two points in $X$ can be joined by a minimal geodesic, which by definition is a continuous path whose length realizes the  distance between its endpoints). 

A metric space $(X,d)$ is called a \emph{partial length space to a point} $p\in X$, if any point $x\in X$, $d(p,x)$ is realized by the infimum of length of all continuous paths connecting $p$ and $x$. By its proof, the Hopf-Rinow-Cohn-Vossen theorem also holds for partial length spaces to a point.

The compactness of a subset $S\subset (X,d)$ can be quantitatively described by the doubling properties, which can be defined either in terms of the least count of balls of half radius that covers a ball in $S$, or the doubling property of a positive Borel measure on $S$. It is well known that the two doubling properties are essentially equivalent; see Volberg-Konyagin \cite{Volberg-Konyagin1988}, Luukkainen-Saksman \cite{Luukkainen-Saksman1998} and Coifman-Weiss \cite{Coifman-Weiss1971}. In the section we will apply the language in doubling measures.

Let $A_0$ be a positive real number.  A metric measured space $(X,d,\mu)$ is called to be \emph{$(\rho,A_0)$-local doubling} if for any $r\le \rho$ and $x\in X$, the measure of $r$-ball $B_r(x)$ satisfies
\begin{equation}\label{local-doubling}
0<\mu(B_r(x)) \le A_0\cdot \mu(B_{\frac{r}{2}}(x))<\infty.
\end{equation}

If $(X,d,\mu)$ is a complete length space with a $(\rho, A_0)$-local doubling measure $\mu$, then $(X,d)$ is locally compact. Thus, by the Hopf-Rinow theorem, $(X,d)$ is a geodesic and proper space. In below we give an uniformly  global doubling estimate on $(X,d,\mu)$. It can be viewed as a quantitative version of the classical Hopf-Rinow theorem.

Let $A:(0,\infty)\to [1.\infty)$ be a positive function.
A metric measured space $(X,d,\mu)$ is called to be globally \emph{$A$-doubling} if for any $r>0$ and $x\in X$, 
\begin{equation}\label{global-doubling}
0<\mu(B_r(x)) \le A(r)\cdot \mu(B_{\frac{r}{2}}(x))<\infty.
\end{equation}

\begin{theorem}[Quantitative Hopf-Rinow-Cohn-Vossen]\label{thm-2-1}
	Given $\rho, A_0>0$, there is a poistive function $A_{\rho,A_0}(r):(0,\infty)\to [1,\infty)$ such that
	any length-metric measured space $(X,d,\mu)$ with $(\rho,A_0)$-local doubling property \eqref{local-doubling} is $A_{\rho,A_0}$-doubling in the sense of \eqref{global-doubling}.
\end{theorem}

The following Gromov-Hausdorff precompactness of $(\rho,A_0)$-local doubling measured metric spaces will be used in the proof of Theorem \ref{thm-2-1} and Sec. \ref{sec-GH-precom}. 

\begin{lemma}\label{lem-2-1}
	Let $(X_i,d_i,\mu_i)$ be a sequence of complete metric measured spaces with $(\rho, A_0)$-local doubling property \eqref{local-doubling}. Suppose that $(X_i,d_i)$ is a partial length space to $p_i\in X_i$. Then there is a subsequence of $(X_i,d_i,p_i)$ that converges to a metric space $(X_\infty, d_\infty, p_\infty)$ in the pointed Gromov-Hausdorff topology.
\end{lemma}

\begin{proof}
	Since $(X_i,d_i)$ is of partial length to $p_i$, every closed ball $\bar B_r(p_i)=\{x\in X_i: d(x,p_i)\le r\}$ coincides with the closure $\overline{B_r(p_i)}$ of open ball. We follow a Hopf-Rinow-typed argument (cf. \cite[Theorem 2.5.28]{Burago-Burago-Ivanov2001}). Let $$R=\sup\{r>0: (\bar B_r(p_i),p_i)\subset (X_i,p_i) \text{ admits a converging subsequence}\}.$$ In order to prove Lemma \ref{lem-2-1}, it suffices to show that $R=\infty$.
	
	First, it is easy to see that $R\ge \frac{\rho}{4}$ by a standard argument. Indeed, for any $0<\epsilon<\frac{\rho}{2}$, let $\{q_j\}_{j=1}^{K}$ be $\epsilon$-discrete points in $B_{\frac{\rho}{4}}(p_i)\subset X_i$. Then all $\frac{\epsilon}{2}$-balls centered at $q_j$,  $B_{\frac{\epsilon}{2}}(q_j)$, are pairwise disjoint, and contained in $B_{\frac{\rho}{2}}(p_i)$. Hence,
	$$\mu(B_{\frac{\rho}{2}}(p_i))\ge \sum_{j=1}^{K}\mu(B_{\frac{\epsilon}{2}}(q_j)).$$ 
	By \eqref{local-doubling}, for the integer $N(2\rho/\epsilon)=[1+\log_2\rho-\log_2\epsilon]$ such that $2^{N-1}\epsilon\in [\frac{\rho}{2},\rho]$,
	$$\mu(B_{\frac{\epsilon}{2}}(q_i))\ge A_0^{-1} \mu (B_\epsilon(q_i))\ge A_0^{-N}\mu(B_{2^{N-1}\epsilon}(q_i)).$$
	Since $B_{2^{N-1}\epsilon}(q_i)$ contains $B_{\frac{\rho}{4}}(p_i)$, each $\mu(B_{2^{N-1}\epsilon}(q_i))$ can be replaced by $\mu(B_{\frac{\rho}{4}}(p_i))$ in the above inequality. Thus, 
	$$\mu(B_{\frac{\rho}{2}}(p_i))\ge K A_0^{-N}\mu(B_{\frac{\rho}{4}}(p_i)),$$
	which together with \eqref{local-doubling} yields
	$\operatorname{Cap}_\epsilon(B_{\frac{\rho}{4}}(p_i))=\max\{K\}\le A_0^{N(\rho/\epsilon)+1}$.
	
	By the Gromov's precompactness, the closed ball $(\bar B_{\frac{\rho}{4}}(p_i),p_i)$ admits a subsequence that converges to a  limit partial length space $(X_\infty, p_\infty)$ to $p_\infty$ in the pointed Gromov-Hausdorff topology.

	Secondly, we show that $R=\infty$ by a contradictory argument. Assume that $R=R_0<\infty$. Let us choose $r>0$ such that $0\le R_0-r<<\frac{\rho}{100}$. Then by the definition of $R_0$, a subsequence of $(B_r(p_i),p_i)$ converges to a pointed compact limit space $(X_\infty, p_\infty)$. For simplicity we assume that $(B_r(p_i),p_i)$ itself converges. 
	
	For any $0<\epsilon\le \frac{\rho}{10}$, let us take finite points $q_j\in X_\infty$ with $0<j\le C(\epsilon; X_\infty)$ whose $\epsilon$-balls cover $X_\infty$, and let $q_{i;j}\in X_i$ be points pointwise close to those $q_j$ in $X_\infty$ in the Gromov-Hausdorff topology. Then by the fact that $X_i$ is a partial length space to $p_i$, $\cup_{j}B_{\frac{\rho}{4}}(q_{i;j})$ covers $\overline{ B_{r+\frac{\rho}{10}}(p_i)}$ for all large $i$. Hence, by the first step, $C(\epsilon;X_\infty)\cdot A_0^{N(\rho/\epsilon)+1}$ of $\epsilon$-balls covers $\overline{B_{r+\frac{\rho}{10}}(p_i)}$. It follows that the $2\epsilon$-packing number
	$\operatorname{Cap}_{2\epsilon}(\overline{ B_{r+\frac{\rho}{10}}(p_i)})\le C(\epsilon;X_\infty)\cdot A_0^{N(\rho/\epsilon)+1}$. By Gromov's precompactness again, $(\overline{ B_{r+\frac{\rho}{10}}(p_i)},p_i)$ admits a convergent subsequence. 
	
	Since $r+\frac{\rho}{10}>R_0$, we meet a contradiction.
\end{proof}

\begin{proof}[Proof of Theorem \ref{thm-2-1}]
	~
		
	Since the local doubling property is preserved after taking completion, without loss of generality we may assume that $X$ itself is complete. 
	
	Let $x$ be an arbitrary fixed point in $X$. 
	We first show by the fact $(X,d)$ is a length space, that for any $R>0$ and $q\in B_R(x)$, 
	\begin{equation}\label{ineq-two-points}
	\mu(B_{\frac{\rho}{2}}(q))\le A_0^{\lceil \frac{2R}{\rho} \rceil}\mu(B_{\frac{\rho}{2}}(x)),
	\end{equation}
	where $\lceil \cdot \rceil$ is to take the ceiling integer.
	 
	Indeed, by the Hopf-Rinow-Cohn-Vossen theorem $X$ is in fact geodesic. Let $\gamma$ be a path with $\gamma(0)=x$ and $\gamma(1)=q$ whose length $= d(x,q)$. Then there are points $p_i=\gamma(t_i)$ such that $p_0=x$, $p_{N}=q$, $d(p_i,p_{i+1})=\frac{\rho}{2}$ for $i<N-1$, $0<d(p_{N-1},q)\le \frac{\rho}{2}$, and $N= \lceil \frac{2R}{\rho} \rceil$ is the ceiling integer. In the following we denote $\mu(p_i,r)=\mu(B_r(p_i))$ for simplicity. Then we have $\mu(p_1,\frac{\rho}{2})\le \mu(x,\rho).$ And by \eqref{local-doubling},
	$$\mu(p_2,\frac{\rho}{2})\le \mu(p_1,\rho)\le A_0\mu(p_1,\frac{\rho}{2})\le A_0\cdot \mu(x,\rho)\le A_0^2 \cdot \mu(x, \frac{\rho}{2}),$$
	By induction,
	$$\mu(q,\frac{\rho}{2})\le A_0^N\cdot \mu(x,\frac{\rho}{2}),$$
	which yields \eqref{ineq-two-points}.

	Next, by Lemma \ref{lem-2-1}, any ball $B(x,r)$ in $X$ admits a uniform number $C(r)=C_{\rho,A_0}(r)$ depends only on $r$, $\rho$ and $A_0$ such that its closure $\overline{B(x,r)}$ is covered by $C(r)$ many $\frac{\rho}{2}$-balls $B_{\frac{\rho}{2}}(q_j)$. Therefore, by \eqref{ineq-two-points} 
	$$\mu(B_r(x))\le \mu(\overline{B_r(x)})\le \sum_{j=1}^{C(r)}\mu(B_{\frac{\rho}{2}}(q_j))\le C(r)A_0^{\lceil \frac{2r}{\rho} \rceil}\cdot \mu(B_{\frac{\rho}{2}}(x)).$$
	
	By taking $A_{\rho,A_0}(r)=C(r)A_0^{\lceil \frac{2r}{\rho} \rceil}$, the proof of Theorem \ref{thm-2-1} is complete.
\end{proof}

\begin{remark}\label{rem-non-length}
	We point out that Lemma \ref{lem-2-1} (and hence Theorem \ref{thm-2-1}) generally fails when $(X_i,d_i)$ is not a partial length space to $p_i$. Example \ref{ex-petersen} provides a counterexample. 
	
	Indeed, for $\lambda_i\searrow1$, let $U_i$ be the open ball $B_{\lambda_i\frac{\pi}{2}}(o_i)\subset (\mathbb S^2,\lambda_i d_0)$, where $(S^2,d_0)$ is in Example \ref{ex-petersen}, and let $\pi:(\widetilde{U}_i,\tilde o_i)\to (U_i,o_i)$ be the universal cover. Let us consider the preimages $\widetilde{W}_i=\pi^{-1}(\overline{ B_{\frac{\pi}{4}}(o_i)})$ of closed balls. Then $(\widetilde{W}_i,d_{\widetilde{U}},\mu_i)$ equipped with the restricted metric $d_{\widetilde{U}_i}$ and the Riemannian volume is $(\frac{\pi}{4},A_0)$-local doubling for constant $A_0>0$. But $(\widetilde{W}_i,d_{\widetilde{U}_i})$ is even not a proper metric space.
\end{remark}

\begin{proof}[Proof of Theorem \ref{thm-lipschitz-type}]
	~
	
	We prove the global doubling property for the completion of $(W,d_W,\mu)$. By applying Theorem \ref{thm-2-1}, it suffices to show that there is a constant $A_{n,\tau}$ depends on $n$ and $\tau$ such that $\overline{(W, d_W)}$ is $(\rho(r_0,t_0,\tau), A_{n,\tau, r_0})$-local doubling equipped with the natural Borel measure and its length metric $d_W$. 
	
	For simplicity, we directly assume $(W,d_W)$ to be its completion. Let $t_0$ be the distance from $\partial W$ toward interior such that for $0<t<t_0$, $\overline{B_{\tau t}(W_t^\circ,W, d_W)}\supset W$. Then
	for any $x\in W$ and $0<t<t_0$, there is $y\in \overline{W_{t}^\circ}$ such that $d_W(x,y)\le \tau t$, and $B_{t+\tau t}(x,W)$ contains a regular ball $B_{t}(y,W)=B_{t}(y, M)$, i.e., the exponential map $\exp_y:B_t(0)\subset T_yM\to B_t(y,W)$ is well-defined. 
	
	Let $t_j=2^{-j}t_0$ and let $j_0>0$ be the minimal integer such that $\tau t_{j_0}\le \frac{r_0}{2}$. For any $0<r\le \rho(r_0,t_0,\tau)=\min \left\{\frac{r_0}{2}, (1+\tau)t_{j_0}\right \}$, there is an integer $j$ such that $(1+\tau)t_j \le \frac{r}{2}<(1+\tau)t_{j-1}$. Let $y_j\in \overline{W_{t_j}^\circ}$ be taken as above.  By \eqref{mthm-1.1} and $r+\tau t_j\le r_0$, $B_{r+\tau t_j}(y_j,M)$ has a compact closure in $M$, such that the Bishop-Gromov's relative volume comparison (\cite{Bishop-Crittenden}, \cite{GLP1981}) holds for $B_{r+\tau t_j}(y_j,M)$ and $B_{t_j}(y_j,M)$, i.e., 
	$$\operatorname{vol}(B_{r+\tau t_j}(y_j,M))\le A_n\left(\frac{r+\tau t_j}{t_j}\right) \operatorname{vol}(B_{t_j}(y_j,M)),$$ where $A_n(\frac{r_1}{r_2})$ is the volume ratio of the $r_1$-ball over the $r_2$-ball in the hyperbolic $n$-space. 
	
	Moreover, $B_{r+\tau t_j}(y_j,W)$ is contained in $B_{r+\tau t_j}(y_j,M)$ and $B_{t_j}(y_j, W)$ coincides with $B_{t_j}(y_j, M)$ as subsets.  Then the Riemannian volume $\mu(r,x)=\mu(B_r(x,W))$ of open ball $B_r(x,W)$ in $(W,d_W)$ satisfies
	$$\mu(r+\tau t_j,y_j)\le \operatorname{vol}(B_{r+\tau t_j}(y_j,M))\le A_n\(\frac{r+\tau t_j}{t_j}\) \mu(t_j,y_j).$$
	By $d_W(x,y_j)\le \tau t_j$ and the triangle inequality, we derive 
	\begin{align*}
	\mu(r,x) &\le \mu(r+\tau t_j,y_j)\le A_n\(\frac{r+\tau t_j}{t_j}\) \mu(t_j,y_j) \\
	&\le A_n\(\frac{2(1+\tau)t_{j-1}+\tau t_j}{t_j}\) \mu(t_j+\tau t_j,x) \qquad\text{ by } \frac{r}{2}<(1+\tau)t_{j-1}\\
	&	\le A_n\(\frac{(4+5\tau)t_j}{t_j}\)\cdot  \mu(\frac{r}{2},x), \qquad \text{ by } (1+\tau) t_j \le \frac{r}{2}
	\end{align*}
	Let $A_{n,\tau, r_0}=\max_{0\le r\le \frac{r_0}{2}}\left\{A_n\(\frac{(4+5\tau)r}{r}\)\right\}$. Then $(W, d_W, \mu)$ is $(\rho(r_0,t_0,\tau), A_{n,\tau, r_0})$-local doubling. By Theorem \ref{thm-2-1} the proof is complete.
\end{proof}

\section{Gromov-Hausdorff precompactness for undistorted domains}\label{sec-GH-precom}

In this section we prove the main Theorem \ref{mthm-1}, which is more general than Theorem \ref{thm-lipschitz-type}. Then we prove Theorems \ref{thm-int} and \ref{thm-cptbdy}.

The proof of Theorem \ref{mthm-1} is motivated by the ideas to exclude those counterexamples mentioned earlier respectively: Example \ref{ex-petersen} by \eqref{mthm-1.1}, and Example \ref{ex-distortion} by \eqref{mthm-1.3}. In principle, once a Gromov-Hausdorff precompactness with an uniform packing number estimate is established locally on any ball of an uniform radius, then a Hopf-Rinow-typed argument implies the desired global precompactness for length spaces. The proofs of Theorems \ref{thm-int} and \ref{thm-cptbdy} are also similar.

In order to avoid confusion, we first recall and fix some notations. As before, an open ball centered at $x$ in a metric space $(X,d_X)$ is denoted by $B_r(x,X)$. An open ball in $(S,d_X)$ with the restricted distance is denoted by $B_r(x,S,d_X)=B_r(x,X)\cap S$. Let $d_S$ be the length metric $d_S$ on $S$ induced by $d_X$. For a subset $E\subset S$, $B_r(E, S)$ is the $r$-neighborhood of $E$ in $(S,d_S)$. 

Recall that, observed by Gromov, the traditional precompactness for $r$-balls in Riemannian manifolds $M$ with a lower Ricci curvature bound is obtained by Bishop-Gromov's relative volume comparison (\cite{Bishop-Crittenden}, \cite{GLP1981}) via shifting the centers of balls, which requires $3r$-balls are relative compact in $M$ (see e.g. \cite[Theorem 2.20]{Perales-Sormani2014}). We need a sharpened version without shifting balls' center as below.
\begin{lemma}\label{lem-balls}
	Let $\{B_{R_0}(p_i,M_i)\}$ be a sequence of $R_0$-balls centered at $p_i$ in (maybe incomplete) Riemannian $n$-manifolds $M_i$, such that 
	\begin{enumerate}
		\item\label{lem-balls-1} $\exp_p:B_{R_0}(o)\to B_{R_0}(p_i,M_i)$ is well-defined, or equivalently, $B_r(p_i,M_i)$ has a compact closure in $M_i$ for any $0<r<R_0$,
		\item\label{lem-balls-2} $\operatorname{Ric}\ge -(n-1)$ on $B_{R_0}(p_i,M_i)$.
	\end{enumerate}
	Then the open ball $(B_{R_0}(p_i,M_i), p_i)$ with its length metric $d_{B_{R_0}(p_i,M_i)}$ is precompact in the pointed Gromov-Hausdorff topology. So is $(B_{R_0}(p_i,M_i), d_{M_i}, p_i)$ equipped with the restricted metric $d_{M_i}$.
\end{lemma}
\begin{proof}
	For simplicity, we denote $U_i=B_{R_0}(p_i,M_i)$ equipped with $d_{U_i}$. 
	Let $W_{i,j}=\bar B_{(1-2^{-j})R_0}(p_i,M_i)$, which by \eqref{lem-balls-1} coincides with the closed ball $\bar B_{(1-2^{-j})R_0}(p_i,U_i)$ as subsets in $M_i$.
	Sine every $x\in W_{i,j}$ is $2^{-j}R_0$-away from $\partial U_i$, $B_r(x, W_{i,j}, d_{U_i})$ coincides with $B_r(x, M_i)$ for any $r<2^{-j}R_0$. Hence $(W_{i,j},d_{U_i})$ is $(2^{-j-2}R_0,A_0)$-local doubling by Bishop-Gromov's relative volume comparison theorem.  Because $(W_{i,j}, d_{U_i})$ is a partial length space to $p_i$, by Lemma \ref{lem-2-1}, $(W_{i,j}, d_{U_i}, p_i)$ admits a converging subsequence for any fixed $j$ as $i\to \infty$.
	
	At the same time, the pointed Gromov-Hausdorff distance between $(W_{i,j}, d_{U_i}, p_i)$ and $(U_i,p_i)$ is no more than $(1-2^{-j})R_0$. By a standard diagonal argument, $(U_i,p_i)$ itself sub-converges in the Gromov-Hausdorff topology. 
	
	Because the identity map from $(B_{R_0}(p_i, M_i), d_{B_{R_0}(p_i,M_i)})$ to $(B_{R_0}(p_i, M_i), d_{M_i})$ $1$-Lipschitz, $\{(B_{R_0}(p_i,M_i), d_{M_i}, p_i)\}$ is also Gromov-Hausdorff precompact.
\end{proof}

\begin{remark}
	In fact, by shifting balls' centers and relative volume comparison, the $\epsilon$-packing number $\operatorname{Cap}_\epsilon(B_{1}(p_i,M_i),d_{B_{1}(p_i,M_i)})$ in Lemma \ref{lem-balls} can be explicitly estimated. We adopt the argument above since the idea is also suitable for more general domains considered below.
\end{remark}
\begin{remark}\label{rem-condition1}
	In Theorem \ref{mthm-1}, \eqref{mthm-1.1} is only used in deriving Lemma \ref{lem-balls}. We will apply the conclusion of Lemma \ref{lem-balls} in the rest proof of Theorem \ref{mthm-1} below.
\end{remark}

\begin{lemma}\label{lem-boundary}
	Let $\{W_i\}$ be a sequence of domains in (maybe incomplete) Riemannian $n$-manifolds $M_i$ satisfying \eqref{mthm-1.1} and \eqref{mthm-1.3}.

	Then for any $x\in W_i$ and any $\epsilon\le 1$, 
	the closed ball $\bar B_{r_0}(x, W_i)$ in $(W_i, d_{W_i})$ can be covered by at most $A_{r_0,t_0,s}(\epsilon)$ $\epsilon$-balls.
\end{lemma}
\begin{proof}
	
	Let $t_j\to 0$ be a monotone decreasing sequence of positive numbers. By the definition of $(s(t),t_0)$-undistortedness \eqref{mthm-1.3}, $s_{j}=s(t_j)\to 0$ as $j\to \infty$, and
	$W_i$ admits an exhaustion of its interior $W_i^\circ=\cup_{j=1}^\infty W_{i,j}$ of closed subsets, where $W_{i,j}=\overline{\left(W_i\right)_{t_j}^\circ}$ is the closure of $t_j$-interior of $W_i$, which satisfies 
	\begin{equation}\label{localest-1} W_{i,j}\subset W_{i,j+1},\;
	d_{W_i}(W_{i,j}, \partial W_i)\ge t_j>0, \text{ and }
	\overline{B_{s_j}(W_{i,j}, d_{W_i})}\supset W_i \text{ for any $j$.}	
	\end{equation}
	
	For any $x_i\in W_i$, we first show that for any fixed $j$,  $\{W_{i,j}\cap B_{r_0}(x_i, W_i)\}$ admits a converging subsequence as $i\to \infty$.

	Indeed, by Lemma \ref{lem-balls}, $B_{r_0}(x_i,M_i)$ is covered by $A_n(\epsilon,r_0)$ of $\frac{\epsilon}{2}$-balls $B_\epsilon(q_{i;k},M_i)$ centered at $q_{i;k}\in B_{r_0}(y_i,M_i)$, where $A_n$ is provided by Lemma \ref{lem-balls}. 
	At the same time, $B_{r_0}(x_i, W_i)$ is contained in $B_{r_0}(x_i, M_i)$.
	By choosing $q_{i;k;j}'\in B_{\frac{\epsilon}{2}}(q_{i;k},M_i)\cap W_{i,j}\cap B_{r_0}(x_i, W_i)$, $W_{i,j}\cap B_{r_0}(x_i,W_i)$ is covered by at most $A_n(\epsilon,r_0)$ of $\epsilon$-balls $B_{\epsilon}(q_{i;k;j},M_i)$ centered in $W_{i,j}\cap B_{r_0}(x_i, W_i)$. 
	
	For any fixed $j$, take $\epsilon\le t_j$. Then by \eqref{localest-1}, for any $z\in W_{i,j}$, $B_{\epsilon}(z,M_i)\cap W_{i,j}=B_{\epsilon}(z, W_i)$ as subsets.  Hence $W_{i,j}\cap B_{r_0}(x_i, W_i)$ is covered by at most $A_n(\epsilon,r_0)$ $\epsilon$-balls $B_{\epsilon}(q_{i;k,j}, W_i)$. By Gromov's precompactness theorem, $B_{r_0}(x_i, W_{i,j}, d_{W_i})$ sub-converges in the Gromov-Hausdorff topology as $i\to \infty$.
	
	Secondly, by \eqref{localest-1} $W_{i,j}\cap B_{r_0}(x_i, W_i)$ is $2s_j$-close to $B_{r_0}(x_i, W_i)$ measured in the Gromov-Hausdorff distance. By a standard diagonal argument, $B_{r_0}(x_i, W_i)$ also sub-converges as $i\to \infty$. 
	
	Now by Gromov's precompactness principle, we conclude Lemma \ref{lem-boundary}.
\end{proof}

\begin{proof}[Proof of Theorem \ref{mthm-1}]
	~
	
	It suffices to consider a sequence of domains $(W_i, d_{W_i}, p_i)\in \mathcal{D}(n,r_0,t_0,s(t))$. 

	Theorem \ref{mthm-1} will follow from a Hopf-Rinow-typed argument similar to the proof of Lemma \ref{lem-2-1}. 
	
	Indeed, let $R=\sup\{r: (\bar B_r(p_i, W_i),p_i) \text{ admits a convergent subsequence}\}$. Then by Lemma \ref{lem-boundary} $R\ge r_0$.
	
	If $R=R_0< \infty$, then a contradiction will be derived from the 2nd part of the proof of Lemma \ref{lem-2-1} together with Lemma \ref{lem-boundary}.
\end{proof}

\begin{remark}
	By the proof of Lemma \ref{lem-boundary}, Theorem \ref{mthm-1} also holds for the following domains with boundary discretely undistorted defined as follows.
	
	Let $t_j$ and $s_j$ be positive numbers such that $t_{j}\to 0$ monotonically, and $s_{j}\to 0$ as $j\to \infty$.	
	Let $W$ be a domain in (maybe incomplete) Riemannian $n$-manifolds $M$, equipped with their length metric $d_W$ and base points $p$. Its boundary $\partial W$ is called \emph{$(t_j,s_j)$-undistored} if $W$ admits an exhaustion of its interior $W^\circ=\cup_{j=1}^\infty W_j$ of closed subsets satisfying 
	$$W_j\subset W_{j+1}, \quad d_W(W_j, \partial W)\ge t_j>0, \quad \text{and} \quad
	\overline{B_{s_j}(W_j, W)}\supset W \text{ for any $j$.}$$ (In particular, $W_j$ is not empty).
\end{remark}

The proof of Theorem \ref{thm-int} is similar to that of Theorem \ref{mthm-1}. 
\begin{proof}[Proof of Theorem \ref{thm-int}]
	~
	
	We will directly prove \eqref{thm-int-4}, which implies \eqref{thm-int-1}.
	Let $M_i$ be a sequence of open Riemannian $n$-manifolds whose $\Ric_{M_i}\ge -(n-1)$. 
	Let $(M_i)_r^\circ$ be the $r$-interior of $M_i$. The by definition of $r$-interior, 
	for any $x\in (M_i)_r^\circ$ and any $0<t<r$, $B_t(x,M_i)$ has a compact closure in $M_i$. 
	
	By Lemma \ref{lem-balls}, for any $0<t\le r$, $B_t(x, (M_i)_r^\circ, d_M)=B_t(x,M)\cap (M_i)_r^\circ$ can be covered by $A_n(\epsilon,t)$ of $\epsilon$-balls $B_\epsilon(q_{k},M_i)$ where $q_k\in (M_i)_r^\circ$.
	
	Hence for $t\le r$, a sequence 
	$(B_t(p_i, (M_i)_r^\circ),d_{M_i})$ of $t$-ball in $((M_i)_r^\circ,d_{(M_i)_r^\circ})$ equipped with the restricted metric $d_{M_i}$ admits a converging subsequence in the Gromov-Hausdorff topology.
	
	Next, we apply a Hopf-Rinow-typed argument to show that the precompactness holds for any $R>0$.
	
	Indeed, let $R=\sup\{t: (B_t(p_i, (M_i)_r^\circ),d_{M_i})\text{ admits a convergent subsequence}\}$.  If $R=R_0<\infty$, then a contradiction will be derived from the 2nd part of the proof of Lemma \ref{lem-2-1}.
	
	For \eqref{thm-int-2}, let $M$ be a manifold in $\mathcal{M}_\partial(r,D,n)$. Since $W\subset M_r^\circ$ and $\operatorname{diam}(W,d_W)\le D$, by the proof of \eqref{thm-int-1}, $(W,d_M)$ is uniformly totally bounded. 
	
	At the same time, $\overline{B_{r}(W,M)}=M$, where by Lemma \ref{lem-balls}, for each $x\in W$, $(B_r(x,M),d_M)$ is also uniformly totally bounded. So is $(M,d_M)$.
	
	For \eqref{thm-int-3}, since each manifold $M\in \mathcal{M}_b(r,D,n)$ has diameter $\le D$, we directly show that for some $\epsilon_i\to 0$, the covering number of $M$ is uniformly bounded.
	
	Indeed, since $\{(\partial M,d_M)\}$ is precompact in the Gromov-Hausdorff topology, for any $\epsilon>0$, there are at most $B(\epsilon)$ $\epsilon$-discrete points $x_k$ in $\partial M$. It follows that $B_{\epsilon/10}(\partial M, \overline{M})$ is contained in $\cup_k B_{2\epsilon}(x_k,\overline M)$, where $\overline{M}$ is the completion of $M$. 
	
	On the other hand, for $\epsilon= 20r_i$, since $\operatorname{diam}(M_{\epsilon/20}^\circ)\le D$, by \eqref{thm-int-1}, $M_{\epsilon/20}^\circ$ can be covered by $A(\epsilon,D)$ $2\epsilon$-balls $B_{2\epsilon}(y_l)$. 
	
	By the fact that $M_{\epsilon/20}^\circ \cup B_{\epsilon/10}(\partial M, \overline{M})=\overline{M}$, $\overline{M}$ can be covered by $A(\epsilon,D)+B(\epsilon)$ $2\epsilon$-balls. Hence $\mathcal{M}_b(r,D,n)$ is precompact.

\end{proof}

\begin{proof}[Proof of Theorem \ref{thm-cptbdy}]
	~
	
	Since $\{(\partial W, d_W, p)\}$ is Gromov-Hausdorff precompact, for any $\epsilon>0$, $B_{r_0/2}(p, \partial W, d_W)$ contains at most $A_0(\epsilon)$ $\epsilon$-discrete points $x_k\in \partial W$.
	It follows that $\epsilon/10$-neighborhood of $B_{r_0/2}(p,\partial W, d_W)$ is contained in $\cup_k B_{2\epsilon}(x_k,W)$.
	
	At the same time, by Lemma \ref{lem-balls} for any $q\in W_{\epsilon/20}^\circ$, $B_{r_0}(q,W_{\epsilon/20}^\circ,d_W)\subset B_{r_0}(q,M)$ can be covered by $A_n(\epsilon,r_0)$ many $\epsilon/20$-balls $B_{\epsilon/20}(y_l,M)$, where $y_l\in W_{\epsilon/20}^\circ$. And thus all balls $B_{\epsilon/20}(y_l,M)$ coincide with $B_{\epsilon/20}(y_l,W)$ as subsets in $M$.
	
	Because $W_{\epsilon/20}^\circ \cup B_{\epsilon/10}(\partial W, W)=W$, let us take $q\in W_{\epsilon/20}^\circ \cap B_{r_0/2}(p,W)$ (when it is nonempty), then the union of $B_{r_0}(q,W_{\epsilon/20}^\circ,d_W)$ and $\cup_kB_{2\epsilon}(x_k,W)$ covers $B_{r_0/2}(p, W)$. Hence $B_{r_0/2}(p,W)$ is uniformly totally bounded. So is $B_{r_0/4}(z,W)$ for any $z\in W$.
	
	Because $(W,d_W)$ is a length space, Theorem \ref{thm-cptbdy} directly follows from a Hopf-Rinow-typed as that for Theorem \ref{mthm-1}.
\end{proof}

\section{Precompactness principle for connected subsets with non-length metrics}\label{sec:precompact-non-length}

In this section we give two precompactness principles for connected subsets with arbitrary boundary equipped with ``$r$-extrinsic''  and ``$\delta$-intrinsic'' metrics respectively, which are defined in below. 

For any subset $X$ in a length space $M$, its \emph{$r$-extrinsic metric} is defined to be the restriction $\left.d_{B_r(X_i,M_i)}\right|_{X}$ of the length metric of its $r$-neighborhood $B_{r}(X,M)$.  
Following Theorem \ref{mthm-1}, it turns out that subsets equipped with their $\delta$-intrinsic metrics definite away from incomplete places of ambient spaces are naturally Gromov-Hausdorff precompact; compare with Theorem \ref{thm-int} (1).

\begin{corollary}\label{cor-1}
	Let $(X_i,p_i)$ be sequence of connected subsets with base points $p_i$ in Riemannian $n$-manifolds $M_i$ whose $\operatorname{Ric}_{M_i}\ge -(n-1)$ and $r_0$-neighborhoods $B_{r_0}(X_i,M_i)$ have complete closures in $M_i$. 
	\begin{enumerate}
		\item\label{cor-1.1} For any $0<r<r_0$, $\left(X_i, p_i\right)$ endowed with its $r$-extrinsic metric is precompact in the pointed Gromov-Hausdorff topology.
		\item\label{cor-1.2} If $\operatorname{diam}\left(X_i,d_{X_i}\right)\le D$ for each $i$, then for $r=r_0$, $\left(X_i,p_i,d_{B_{r_0}(X_i,M_i)}\right)$ is also precompact in the pointed Gromov-Hausdorff topology.
	\end{enumerate}
	
\end{corollary}
\begin{proof}
	\eqref{cor-1.1} follows directly from Theorem \ref{mthm-1}, because $(X_i, d_{B_r(X_i,M_i)})$ is a subspace of the $r$-neighborhood $B_r(X_i,M_i)$, which is a $(t,r)$-undistorted domain of $1$-Lipschitz, whose $(r_0-r)$-neighborhood is complete in $M_i$.
	
	For \eqref{cor-1.2}, let us consider $B_{r_0/2}(X_i,M_i)$, whose intrinsic diameter with respect to its length metric $\le D+r_0$. Then $W=B_{r_0/2}(X_i,M_i)$ is contained in $r_0/2$-interior of $(B_{r_0}(X_i,M_i), d_{B_{r_0}(X_i,M_i)})$ such that $\overline{B_{r_0/2}(W,M_i)}=B_{r_0}(X_i,M_i)$. By viewing $W=B_{r_0/2}(X_i,M_i)$ and $M=B_{r_0}(X_i,M_i)$ in \eqref{thm-int-2}, it directly follows from Theorem \ref{thm-int} that $(B_{r_0}(X_i,M_i), d_{B_{r_0}(X_i,M_i)})$ is uniformly totally bounded. So is the subspace $(X_i,d_{B_{r_0}(X_i,M_i)})$.
\end{proof}

\begin{remark}
Note that $r$ cannot equal $r_0$ in \eqref{cor-1.1}. For example, let  $$X_i=\pi^{-1}\left(B_{\frac{\pi}{2}-r_0-2^{-i}}(o)\right)$$ be as in Example \ref{ex-petersen} equipped with the restricted metric $d_{B_{r_0}(X_i,\widetilde{U})}$, where the $\widetilde{U}$ is the universal cover  of $B_{\frac{\pi}{2}}(o) \subset (\mathbb S^2, d_0)$. Then $\left\{\left(X_i, \tilde o, d_{B_{r_0}(X_i,\widetilde{U})}\right)\right\}$ is not precompact. It also illustrates that the diameter bound of $(X_i,d_{X_i})$ is necessary in \eqref{cor-1.2}.	
\end{remark}

One partial motivation to consider $r$-extrinsic metrics is to improve an almost warped product principle  \cite[Theorem 3.6]{Cheeger-Colding1996} by Cheeger-Colding for open annuli in complete Riemannian manifolds with a lower Ricci curvature bound in the sense of Gromov-Hausdorff topology, whose special cases are the almost splitting theorem and ``almost volume cone implies almost metric cone'' fundamentally and widely applied in the study of manifolds with a lower Ricci curvature bound. 

Let us briefly recall \cite[Theorem 3.6]{Cheeger-Colding1996}. Let $(M^n,g)$ be a complete Riemannian $n$-manifold with Ricci curvature $\operatorname{Ric}_{(M,g)}\ge (n-1)\Lambda>-\infty$. Let $K \subset M^n$ be a compact subset. Let $r(x)=d(x,K)$ denote the distance function from $K$ and for $0 < a < b$, let $A_{a,b} = r^{-1} ((a, b))$ be the annulus. For simplicity we assume $A_{a,b}$ is connected. 

Let $f:[a,b]\to \mathbb R$ be a non-negative function and let $\mathcal{F}(r) =-\int_r^bf(u)du$. It is a classical result in Riemannian geometry (cf. \cite[\S I.1]{Cheeger-Colding1996}) that for any $a < c < d < b$, $r^{-1}([c, d])$ is isometric to a warped product $$[c,d]\times_f r^{-1}(c)=\left([c,d]\times r^{-1}(c), dr^2+f^2\cdot f(c)^{-2}g|_{r^{-1}(c)}\right)$$ if and only if $\operatorname{Hess}\mathcal F=f'(r)g$ on $A_{a,b}$. 
 
 Note that, the distance between $(r_1, x_1), (r_2, x_2)$ in a warped product $[a,b]\times_fX$ is given by a function $\rho_f(r_l, r_2, d_X(x_1,x_2))$. Equivalently,  for $\underline y_1=(r_1,x)$, $\underline y_2=(r_2,x)$ and $\underline z_1=(r_1',x')$, $\underline z_2=(r_2',x')$, then distance between points $\underline{y}_2, \underline{z}_2$ satisfies
 \begin{equation}\label{eq-cosine-law}
 	d(\underline{y}_2,\underline{z}_2)=Q\left(r_1,r_1',r_2,r_2',d(\underline{y}_1,\underline{z}_1)\right),
 \end{equation} 
 which is uniquely determined by $r_1,r_1',r_2,r_2',d_X(x,x')$. In the following we call \eqref{eq-cosine-law} the cosine law for $[a,b]\times_f X$.

It was shown in \cite[Theorem 3.6]{Cheeger-Colding1996} that if $\operatorname{Hess}\mathcal F=f'(r)g$ holds in the following $L^1$-sense \eqref{cond-L1-warped}, then the annulus $A_{a,b}$ approximates to a warped product with function $f$ in the Gromov-Hausdorff distance.

 Suppose there is a function $\underline{\mathcal F}:A_{a,b}\to \mathbb R$ satisfying 
\begin{equation}\label{cond-L1-warped}
	\begin{aligned}
\underline{\mathcal F}(A_{a,b})\subset \mathcal F\circ r(A_{a,b}),\qquad \left|\underline{\mathcal F}-\mathcal F\circ r\right|\le \delta,\\
\frac{1}{\operatorname{vol}(A_{a,b})}\int_{A_{a,b}}\left|\nabla \underline{\mathcal F}-\nabla \mathcal F\circ r\right|\operatorname{dvol}\le \delta,\quad \text{and}\\
\frac{1}{\operatorname{vol}(A_{a,b})}\int_{A_{a,b}}\left|\operatorname{Hess}\underline{\mathcal F}- f'(\mathcal F^{-1}(\underline{\mathcal F}))g\right|\operatorname{dvol}\le \delta		
	\end{aligned}
\end{equation}

 For fixed $u>0$, let $$\mathcal{V}(u)=\inf\left\{\frac{\operatorname{vol} (B_u(q))}{\operatorname{vol}(A_{a,b})}: \text{ for all } q\in A_{a,b} \text{ with } u< \min\{b - r(q), r(q) - a\}\right\}.$$
 
 By the segment inequality (\cite[Theorem 2.11]{Cheeger-Colding1996}), the following approximated cosine law of warped product was proved in \cite{Cheeger-Colding1996}.
 \begin{lemma}[{\cite[Proposition 2.80]{Cheeger-Colding1996}}]\label{lem-cosine-law}
 	Given $\epsilon> 0$ and some $p\in A_{a,b}$ with
 	$B_{4R}(p) \subset A_{a,b}$, there exists $\zeta(\epsilon,n,\mathcal V,r(p)-a,b-r(p))$, such that if \eqref{cond-L1-warped} holds for $\delta<\zeta$ then for $x_1, z_1,x_2,z_2\in B_R(p)$ such that $r(z_i)-r(x_i)=d(x_i,z_i)$ ($i=1,2$), 
 	 \begin{equation}\label{ineq-cosine-law}
 	 \left|d\left(z_1,z_2\right)-Q\left(r(x_1),r(z_1),r(x_2),r(z_2),d(x_1,z_1)\right)\right|<\epsilon,
 	 \end{equation}
 	 where $Q$ is the distance between points $\underline{x}_2=(r(x_2), \theta), \underline{z}_2=(r(z_2),\theta')$ in the model warped product $(a,b)\times_f\mathbb R$ given in \eqref{eq-cosine-law}, and $\theta-\theta'$ is chosen so that $d_{(a,b)\times_f\mathbb R}(\underline{x}_1,\underline{z}_1)=d(x_1,z_1)$ with $\underline{x}_1=(r(x_1),\theta)$, $\underline{z}_1=(r(z_1),\theta')$.
 	  \end{lemma}

In order to show for $\alpha > 0$, an annulus $A_{a+\alpha,b-\alpha}$ is close in the Gromov-Hausdorff distance to a warped product $(a + \alpha, b - \alpha)\times_f X$, slightly larger annuli $A_{a+\alpha',b-\alpha'}$ for 
$0\le \alpha'<\alpha$ are chosen and the restricted metric $d^{\alpha',\alpha}=\left.d_{A_{a + \alpha', b - \alpha'}}\right|_{A_{a + \alpha, b - \alpha}}$ on $A_{a + \alpha, b - \alpha}\subset A_{a + \alpha', b - \alpha'}$ is considered. Let $\underline{d}^{\alpha',\alpha}$ denote the corresponding restricted metric on $(a+\alpha,b-\alpha)\times_f X$ from $(a+\alpha',b-\alpha')\times_f X$. Let $\Psi(u_1,.\dots,u_k|.,\dots)$ denote a nonnegative function depending on the numbers, $u_1,\dots u_k$, and some additional parameters, such that when those additional parameters are fixed, one has $\lim_{u_1,\dots,u_k\to 0} \Psi(u_1,\dots,u_k|.,\dots)=0$.

The following improved almost rigidity of warped product follows from Lemma \ref{lem-cosine-law} and Corollary \ref{cor-1}.

\begin{theorem}[{compared with \cite[Theorem 3.6]{Cheeger-Colding1996}}]\label{thm-almost-rigidity}
Let $\alpha-\alpha'>\xi>0$. Assume that \eqref{cond-L1-warped} holds for $\delta<\zeta$, and
\begin{enumerate}
	\item \label{thm-almost-rigidity-1} $\operatorname{diam}(A_{a+\alpha,b-\alpha},d^{\alpha',\alpha}) \le D,$
	\item\label{thm-almost-rigidity-2} for all $x\in r^{-1}(a + \alpha')$, there exists $y \in  r^{-1}(b - \alpha')$ with
	$$ d^{\alpha'-\zeta}(x,y) 
	\le b- a -2\alpha' + \zeta,$$
\end{enumerate}
Then there exists a length space $X$, with $\operatorname{diam}(X) \le c(a,b,\alpha',f ,D)$, such that for the
metrics $d^{\alpha',\alpha}, \underline{d}^{\alpha',\alpha}$
$$ d_{GH}\left(A_{a+\alpha,b-\alpha},(a + \alpha,b - \alpha) \times_f X)\right) \le \Psi(\zeta|\alpha',\xi,n,f,D,\mathcal V).$$	
\end{theorem}
\begin{proof}
 Since $A_{a+\alpha',b-\alpha'}\supset B_{\alpha-\alpha'}(A_{a+\alpha,b-\alpha})$, the $(\alpha-\alpha')$-extrinsic metric on $A_{a+\alpha,b-\alpha}$ satisfies
 \begin{equation}\label{ineq-almost-rig-1}
 	d_{B_{\alpha-\alpha'}(A_{a+\alpha,b-\alpha})}(x,y)\ge d^{\alpha',\alpha}(x,y), \qquad \text{for any $x,y\in A_{a+\alpha,b-\alpha}$}.
 \end{equation} On the other hand, by \eqref{thm-almost-rigidity-2} and the triangle inequality, it is easy to see $A_{a+\alpha',b-\alpha'}\subset B_{\alpha-\alpha'+\zeta}(A_{a+\alpha,b-\alpha})$. Hence for fixed $\zeta_0=\alpha-\frac{1}{2}\alpha'>\zeta$
 \begin{equation}\label{ineq-almost-rig-2}
 	d_{B_{\alpha-\alpha'+\zeta_0}(A_{a+\alpha,b-\alpha})}(x,y)\le d^{\alpha',\alpha}(x,y), \qquad \text{for any $x,y\in A_{a+\alpha,b-\alpha}$}.
 \end{equation} 
 Applying Corollary \ref{cor-1} to the $(\alpha-\alpha')$-extrinsic and $(\alpha-\frac{1}{2}\alpha')$-extrinsic metrics on $A_{a+\alpha,b-\alpha}$, we conclude that the Gromov-Hausdorff precompactness holds for all $(A_{a+\alpha,b-\alpha},d^{\alpha',\alpha})$ with $\Lambda$ and $n$ fixed (the precompactness also follows from Theorem \ref{thm-int} (1), since $A_{a+\alpha,b-\alpha}$ lies in the $(\alpha-\alpha')$-interior of $A_{a+\alpha',b-\alpha'}$).
 
 Let $A^\infty_{a+\alpha,b-\alpha}$ be the Gromov-Hausdorff limit of any convergent subsequence $(A_{a+\alpha,b-\alpha},d^{\alpha',\alpha})$ satisfying \eqref{thm-almost-rigidity-1} as $\zeta\to 0$. Then the cosine law \eqref{ineq-cosine-law} of warped product $(a,b)\times_f \mathbb R$ holds as an equality on every $B_R(p)\subset B_{4R}(p)\subset A^\infty_{a,b}$. So is for \eqref{thm-almost-rigidity-2}. Let $r_\infty$ be the limit of $r$, and $X$ be $r_\infty^{-1}(a+\alpha)$ equipped with the length metric induced from $A^{\infty}_{a+\alpha,b-\alpha}$ rescaled by multiplying $\frac{1}{f(a+\alpha)}$.
 Then by the cosine law \eqref{ineq-cosine-law} and geodesic extension property \eqref{thm-almost-rigidity-2} without error term, it is direct to verify the map $$\pi:A^{\infty}_{a+\alpha,b-\alpha}\to [a+\alpha,b-\alpha]\times_f X,\qquad \pi(z)=(r_\infty(x), y),$$ where $y$ is a closest point in $r^{-1}(a+\alpha)$ to $z$ is any isometry.
 
 Now the conclusion follows from a standard argument by contradiction.
\end{proof}
\begin{remark}
	The hypothesis of Theorem \ref{thm-almost-rigidity} is the same as \cite[Theorem 3.6]{Cheeger-Colding1996}, and the only difference is that, in Theorem \ref{thm-almost-rigidity} $X$ is a length space, which was unknown in \cite[Theorem 3.6]{Cheeger-Colding1996} 
	due to the lack of Gromov-Hausdroff precompactness principle for annuli.  It was proved in \cite{Cheeger-Colding1996} that $X$ can be chosen a length space under the following stronger condition
	$$\frac{\operatorname{vol}(A_{a,b})}{\operatorname{vol}(r^{-1}(a))}\ge (1-\omega)\frac{\int_a^bf^{n-1}(u)du}{f^{n-1}(a)} \quad\text{with}\quad
	\omega=\omega(\zeta,\alpha',n,a,b,f,\mathcal{V}),$$
	which in fact implies \eqref{thm-almost-rigidity-1},  \eqref{thm-almost-rigidity-2} and the Gromov-Hausdorff precompactness.
	
	The proof of Theorem \ref{thm-almost-rigidity} is also simpler than that of \cite[Theorem 3.6]{Cheeger-Colding1996}, which involved the error term in the cosine law \eqref{ineq-cosine-law} and \eqref{thm-almost-rigidity-2}. In comparison, the error term are removed by taking limit in the above proof of Theorem \ref{thm-almost-rigidity}.
\end{remark}

In the original proof of \cite[Theorem 3.6]{Cheeger-Colding1996}, $X$ was constructed as a level set $r^{-1}(a+\alpha')$ equipped with its $\delta$-intrinsic metric (see below) multiplying the rescaled constant $\frac{1}{f(a+\alpha')}$ with $0<\delta<\alpha'$. 

For any subset $Y$ in a length space $M$, the \emph{$\delta$-intrinsic metric} on $Y$ is defined by
\begin{equation*}
	d^\delta_Y(x,y)=\inf\left\{\sum_{k=0}^{N} d(x_k,x_{k+1}): x_k\in Y, x_0=x, x_N=y,  d_M(x_k,x_{k+1})\le \delta\right\}.
\end{equation*}

In comparison, the length metric of $r^{-1}(a+\alpha')$ generally admits no uniform local control. Then it was verified in \cite{Cheeger-Colding1996} that the projection $\pi:A_{a+\alpha',b-\alpha'}\to [a+\alpha',b-\alpha']\times_f X$, $\pi(y)=(r(y),z)$ with $z$ a closest point in $r^{-1}(a+\alpha')$ to $y$ is a Gromov-Hausdorff approximation after restricted to $A_{a+\alpha,b-\alpha}$. 

By the proof of \cite[Theorem 3.6]{Cheeger-Colding1996} and Theorem \ref{thm-almost-rigidity}, under the hypothesis of Theorem \ref{thm-almost-rigidity}, $\left(r^{-1}(a+\alpha'),\frac{1}{f(a+\alpha')}d_{r^{-1}(a+\alpha')}^\delta\right)$ approximates to a length space as $\zeta,\delta\to 0$, where $\zeta$ also depends on $\delta$.

In the end of this section, we prove the following precompactness principle for $\delta$-intrinsic metrics in general.
\begin{corollary}\label{cor-2}
	Let $\{(X_i,p_i)\}$ be a sequence of connected subsets with base points $p_i$ in (maybe incomplete) Riemannian $n$-manifolds $M_i$ such that 
	$$r_1=\inf_i\sup \left\{r>0:\begin{aligned}
		& B_{r}(X_i,M_i) \text{ has a complete closure in $M_i$,}\\
		&\operatorname{Ric}_{B_{r}(X_i,M_i)}\ge - (n-1)	
	\end{aligned} \right\}>0.$$
	Then for any $0<\delta < 2r_1$, the family $\left(X_i, d_{X_i}^\delta, p_i\right)$ is precompact in the pointed Gromov-Hausdorff topology.
\end{corollary}
Before given the proof, let us first point out that the restriction $\delta< 2r_1$ on $d_{X_i}^\delta$ is optimal for Corollary \ref{cor-2}.  For a counterexample, let $X_i=\pi^{-1}(\overline{B_{\frac{\pi}{2}-r_1}(o)})\subset \widetilde{U}$ be as in Example \ref{ex-petersen}, where $\pi:(\widetilde{U},\tilde o)\to (U,o)$ is the universal cover and $U=B_{\frac{\pi}{2}}(o) = (\mathbb S^2, d_0)\setminus\{p,p^*\}$. Then on $X_i\subset U$, \eqref{mthm-1.1} holds for any $0<r_0<r_1$, but fails for $r_1$. If one takes $\delta=2r_1$, then any closed $\delta$-ball $\overline{B_{\delta}(z_i)}$ at $z_i\in \partial X_i$ in $(X_i,d^{\delta}_{X_i})$ contains the whole boundary $\partial X_i$, which is not totally bounded.

\begin{proof}
	It suffices to show that for any $R>0$, and any $0<\epsilon\le \delta$, there are at most $C(\epsilon,R)$ many $\epsilon$-balls whose union covers $B_R(p_i, X_i, d_{X_i}^\delta)$. 
	
	Let us take $W_i=\overline{B_{\frac{\delta}{2}}(X_i, M_i)}$, the closed $\frac{\delta}{2}$-neighborhood of $X_i$ in $(M_i,d_{M_i})$. Then by $\delta< 2r_1$, we see that  $\frac{2r_1-\delta}{4}$-neighborhood of $W_i$ has a complete closure in $M_i$. By Theorem \ref{mthm-1}, $(W_i,d_{W_i},p_i)$ is precompact in the Gromov-Hausdorff topology.
	
	The benefit of $W_i$ is that, the distance $d_{X_i}^\delta$ is realized by the length of curves in $W_i$. This can be seen directly from the definition of $d_{X_i}^\delta$. Hence, we have for any two points $x,y\in X_i$, $d_{X_i}^\delta(x,y)\ge d_{W_i}(x,y)$, and any $R$-ball $B_R(p_i, X_i, d_{X_i}^\delta)$ is contained in $B_R(p_i, W_i)$. 
	
	Now a uniform upper bound on the $\epsilon$-covering number of $B_R(p_i, X_i, d_{X_i}^\delta)$ can be easily derived by the following two facts.
	
	The Gromov-Hausdorff precompactness of $(W_i,d_{W_i},p_i)$ implies that for any $R>0$ and $\epsilon\le \delta$, there are at most $C(\epsilon,R)$ balls $B_{\epsilon/2}(q_k, W_i)$ whose union covers $B_R(p_i, W_i)$. Hence the subset $B_R(p_i,X_i,d_{X_i}^\delta)$ can be covered at most $C(\epsilon,R)$ balls $B_{\epsilon}(q_k', X_i, d_{W_i})$ at $q_k'\in B_{\epsilon/2}(q_k,W_i)\cap X_i$ whenever it is non-empty. 
	
	At the same time, for $x,y\in X_i$, if $d^\delta_{X_i}(x,y)\le \delta$ or $d_{W_i}(x,y)\le \delta$, then by definition of $\delta$-intrinsic metric and $W_i$, $d^\delta_{X_i}(x,y)=d_{M_i}(x,y)=d_{W_i}(x,y)$. Hence $B_{\epsilon}(q_k', X_i, d_{W_i})$ coincides with $B_{\epsilon}(q_k', X_i, d^\delta_{X_i})$.
\end{proof}

\section{A simplified proof of semilocally simply connectedness of a Ricci-limit space}\label{sec:semilocal-sim}
In this section we give a simplified proof of a key fact \eqref{semilocal-sim-1} below in proving the semilocally simply connected property of a Ricci-limit space.

\begin{theorem}[Pan-Wei \cite{Pan-Wei2019}, Wang \cite{Wang2021}]\label{thm-semilocal-sim}
	Let $X$ be a Ricci-limit space, i.e., there is a sequence $(M_i,p_i)$ of complete Riemannian $n$-manifolds with $\Ric\ge -(n-1)$ that converges to $(M_i,p_i)$ in the pointed Gromov-Hausdorff topology. Then $X$ is semi-locally simply connected. That is, for any $x\in X$, there is $r(x)>0$ such that any loop in $B_{r(x)}(x,X)$ is contractible in $X$.
\end{theorem}

In the above theorem, it was first proved by Pan-Wei \cite{Pan-Wei2019} for non-collapsed limits. Then based on the ideas in \cite{Pan-Wei2019} and techniques built by Pan-Wang \cite{Pan-Wang2021}, it was proved by Wang \cite{Wang2021} for the general case including collapsed Ricci-limit spaces.

The key observations in \cite{Wang2021} are as follows.
\begin{enumerate}
	\item\label{semilocal-sim-1} For any $x\in \bar B_{1/2}(p)$ and any $0<l<1/2$, there is a small $r(x,l)>0$ such that any loop $\gamma$ at $y\in B_r(x)$ admits approximations loops $\gamma_i$ at $y_i\in M_i$ up to a subsequence, which is homotopic to loops $\gamma_i'$ whose length $\le \epsilon_i\to 0$ by a homotopy whose image lies in $B_{4l}(x_i)$;
	\item\label{semilocal-sim-2} The homotopy between $\gamma_i$ and $\gamma_i'$ gives rise to a triangular decomposition $\Sigma$ of the unit disk $D^2$ with a continuous map $H:K^1\to B_{5l}(x)$ from the $1$-skeleton of $\Sigma$ such that for each $\triangle$ of $\Sigma$, $\operatorname{diam}(\triangle)<\epsilon'_i\to 0$, $\operatorname{diam}H(\partial \triangle)\le 100\epsilon_i$ and $H|_{\partial D}=\gamma$.
\end{enumerate}

Then following Pan-Wei \cite{Pan-Wei2019} (see also \cite[Lemma 3.3]{Wang2021}),  the iterated decomposition provided by \eqref{semilocal-sim-2} for each small loops uniformly converges to a continuous homotopy between $\gamma$ and a trivial loop in $\bar B_{1/3}(p)$; for details see \cite{Wang2021}.

The original proof of \eqref{semilocal-sim-1} by Wang \cite{Wang2021} involves a nontrivial technical tool, i.e. the slice theorem of pseudo-group actions on a Ricci-limit space proved by Pan-Wang \cite{Pan-Wang2021}.
Since the universal cover $\widetilde{B_l(x_i)}$ of $B_l(x_i)$ does not converge in the pointed Gromov-Hausdorff topology,  the pseudo-group actions $G_i$ on the $l$-ball $B_l(\tilde x_i)$ in $\widetilde{B_l(x_i)}$ was considered, where $G_i=\{g\in \Gamma_i \;|\; d(g\tilde x_i,\tilde x_i)\le 1/100\}$ is the restriction of the deck-transformation $\Gamma_i$ on the universal cover $\widetilde{B_l(x_i)}$. By passing to a subsequence, one has
$$\begin{CD}
	(\bar B_l(\tilde x_i), \tilde x_i, G_i) @>GH>> (\bar B_l(\tilde x), \tilde x, G)\\
	@VVV @VVV\\
	(\bar B_l(x_i),x_i) @>GH>> (\bar B_l(x),x)
\end{CD}$$

As an extension of the slice theorem of a proper Lie group action by Palais \cite{Palais1961}, it was proved in Pan-Wang \cite{Pan-Wang2021} that the limit pseudo-group $G$ admits a slice $S$ at $\tilde x$, i.e., a subset $S\subset B_l(\tilde x)$ such that the followings hold.
\begin{enumerate}
	\item $\tilde x\in S$ and $S$ is preserved by the isotropy group $G_{\tilde x}=\{g\in G\;|\; g\tilde x=\tilde x\}$;
	\item the $G$-orbit $GS$ along $S$ is open, and $[g,s]\mapsto gs$ is a $G$-homeomorphism between $G\times_{G_{\tilde x}}S$ and $GS$. In particular, $S/G_{\tilde x}$ is homeomorphic to $GS/G$.
\end{enumerate}
By taking $B_r(x)$ is contained in $GS/G$, any loop at $y$ in $B_r(x)$ admits a lifting path $\tilde \gamma$ in $S$. Then approximated path $\tilde \gamma_i$ in $B_{4l}(\tilde x_i)$ can be constructed, which is homotopic to a small path $\tilde \gamma_i'$ of length $\le 2\epsilon_i$ in the universal cover of $\bar B_{4l}(x_i)$, whose projections are loops $\gamma_i$ and $\gamma_i'$ together with the homotopy lying in $B_{4l}(x_i)$.

We now point out that,
instead of the pseudo-group actions, the equivariant precompactness provided by Corollary \ref{cor-pre} yields a simple and direct proof of \eqref{semilocal-sim-1}.

Indeed, let us consider the equivariant convergence
$$\begin{CD}
	(\widehat{B}(x_i),l,4l), \hat x_i, \hat \Gamma_i(x_i,4,4l)	@>GH>> (Y_{l,4l}, \hat x, G)\\
	@VVV @VVV\\
	(B_{l}(x_i),x_i) @>GH>> (B_l(x),x)
\end{CD}$$
Then by Colding-Naber \cite{Colding-Naber2012}, $G$ is a Lie group. By the standard slice theorem by Palais \cite{Palais1961},
a definite small loop can be lift to a path in a slice in $Y_{l,4l}$. Then the approximated paths in $\widehat{B}(x_i,l,4l)$ is homotopic to a new path with length $\le \epsilon_i$ in $\widetilde{B_{4l}(x_i)}$, both of which project to loops and homotopy in $B_{4l}(x_i)$.

\begin{remark}
	Another proof of Theorem \ref{thm-semilocal-sim} can be found in Wang \cite{Wang2023}, which relies on the stability of relative $\delta$-cover over a ball in a Ricci-limit space by Sormani-Wei \cite[Theorem 3.12]{Sormani-Wei2004} (cf. Mondino-Wei \cite[Theorem 4.5]{Mondino-Wei2019} for $\operatorname{RCD}^*(K,N)$-spaces). This proof does not depend on the property \eqref{semilocal-sim-1}.
\end{remark}


\begin{thebibliography}{10}
	\bibitem{Adams-Fournier1977}{R. Adams and J. Fournier}, {\em Cone conditions and properties of Sobolev spaces}. J. Math. Anal. Appl. {\bf 61} (1977), no. 3, 713-734.
	
	\bibitem{AGS}{L. Ambrosio, N. Gigli and G. Savar\'u}, {\em Metric measure spaces with Riemannian Ricci curvature bounded from below}. Duke Math. J. {\bf 163} (2014), no. 7, 1406-1490.
	
	\bibitem{Anderson-Katsuda-Kurylev-Lassas-Taylor2004}{M. Anderson, A. Katsuda, Y. Kurylev, M. Lassas and M. Taylor}, {\em 
	Boundary regularity for the Ricci equation, geometric convergence, and Gelfand’s inverse boundary problem}. Invent. Math. {\bf 158} (2004), 261–321.
	
	\bibitem{Bishop-Crittenden}{R. L. Bishop and R. J. Crittenden}, {\em Geometry of Manifolds}. Academic Press, New York, 1964.
	
	\bibitem{Burago-Burago-Ivanov2001}{D. Burago, Y. Burago and S. Ivanov}, {\em A course in metric geometry}. Volume 33 of Graduate Studies in Mathematics. American Mathematical Society, Providence, RI, 2001.
	
	\bibitem{Burago-Gromov-Perelman1992}{Y. Burago, M. Gromov and G. Perelman}, {\em A. D. Alexandrov spaces with curvature
	bounded below}. Uspekhi Mat. Nauk {\bf 47} (1992), 3–51.
	
	\bibitem{Cheeger-Colding1996}{J. Cheeger and T. H. Colding}, {\em Lower bounds on Ricci curvature and the almost rigidity of warped Products}. Ann. of Math. {\bf 144} (Jul., 1996), no. 1, 189-237.
	
	\bibitem{Cheeger-Colding1997}{J. Cheeger and T. H. Colding}, {\em On the structure of space with Ricci curvature bounded below I}. J. Differential Geom. {\bf 46} (1997), 406-480.
	
	\bibitem{Cheeger-Fukaya-Gromv}{J. Cheeger, K. Fukaya and M. Gromov}, {\em Nilpotent structures and invariant metrics on collapsed
	manifolds}. J. Amer. Math. Soc. {\bf 5} (1992), no .2, 327–372.
	
	\bibitem{Cheeger-Gromov1986}{J. Cheeger and M. Gromov}, {\em Collapsing Riemannian manifolds while keeping their curvature
	bounded. I.} J. Differential Geom. {\bf 23} (1986), no. 3, 309–346.
	
	\bibitem{Cheeger-Gromov1990}{J. Cheeger and M. Gromov}, {\em Collapsing Riemannian manifolds while keeping their curvature
	bounded. II}. J. Differential Geom. {\bf 32} (1990), no. 1, 269–298.
	
	\bibitem{Chen-Rong-Xu2016}{L. Chen, X. Rong and S. Xu}, {\em Quantitative space form rigidity under lower Ricci curvature bound I}. J. Differential Geom. {\bf 113} (2019), no. 2, 227-272.
	
	\bibitem{Chen-Rong-Xu2018}{L. Chen, X. Rong and S. Xu}, {\em Quantitative space form rigidity under lower Ricci curvature bound II}. Trans. Amer. Math. Soc. {\bf 370} (2018), no. 6, 4509-4523.
	
	\bibitem{Cohn-Vossen1936}{S. Cohn-Vossen}, {\em Existenz k\"urzester wege}. Compositio Math., Groningen, {\bf 3} (1936), 441-452.
	
	
	\bibitem{Coifman-Weiss1971}{R. R. Coifman and G. Weiss}, {\em Analyse harmonique non-commutative sur certains espaces homog\`enes}. Lecture Notes in Math., vol. 242, Springer-Verlag, Berlin, 1971.
	
	\bibitem{Colding-Naber2012}{T. H. Colding and A. Naber}, {\em Sharp H\"older continuity of tangent cones for spaces with a lower Ricci curvature bound and applications}. Ann. of Math. (2) {\bf 176} (2012), 1173–1229.
	
	\bibitem{Fukaya1988}{K. Fukaya}, {\em A boundary of the set of the Riemannian manifolds with bounded curvatures and diameters}. J. Differential Geom. {\bf 28} (1988), 1–21.
	
	\bibitem{Fukaya-Yamaguchi1992}{K. Fukaya and T. Yamaguchi}, {\em
		The Fundamental Groups of Almost Nonnegatively Curved Manifolds}. Ann. of Math. {\bf 136}  (1992), 253-333.
	
	\bibitem{Gilbarg-Trudinger2001}{D. Gilbarg and N. S. Trudinger}, {\em Elliptic Partial Differential Equations of Second Order}. Springer, Berlin. 2001.
	
	\bibitem{GLP1981}{M. Gromov}, {\em Structures m\`etriques pour les vari\`et\`es riemanniennes}. (French) [Metric structures for Riemann manifolds]. Textes Math\'ematiques [Mathematical Texts], 1. CEDIC/Fernand Nathan,	Paris, 1981.
	
	\bibitem{Hopf-Rinow1931}{H. Hopf and W. Rinow}, {\em Ueber den Begriff der vollst\"andigen differentialgeometrischen Fl\"achen}. Comm. Math. Helv. {\bf 3} (1931), pp. 209-225.
	
	\bibitem{Huang-Kong-Rong-Xu2020}{H. Huang, L. Kong, X. Rong and S. Xu}, {\em Collapsed manifolds with Ricci bounded covering geometry}. Trans. Amer. Math. Soc. {\bf 373} (2020), no. 11, 8039-8057.
	
	\bibitem{Jiang-Kong-Xu}{Z. Jiang, L. Kong and S. Xu}, {\em Convergence of Ricci-limit spaces under bounded Ricci curvature and local covering geometry}. preprint (2022), 	arXiv:2205.00373.
	
	
	
	\bibitem{Kapovitch-Wilking2011}{V. Kapovitch and B. Wilking}, {\em Structure of fundamental groups of manifolds with Ricci curvature bounded below}. preprint (2011).
	
	\bibitem{Knox2012}{K. Knox}, {\em A compactness theorem for Riemannian manifolds with boundary and applications}.  arXiv:1211.6210
	
	\bibitem{Kodani1990}{S. Kodani}, {\em Convergence theorem for Riemannian manifolds with boundary}. Compositio Math. {\bf 75} (1990) 171–192.

	\bibitem{Kosovskii}{N. N. Kosovski\v i}, {\em Gluing of Riemannian manifolds of curvature $\ge k$}. Algebra i Analiz {\bf 14} (2002), 140–157 In Russian; translated in St. Petersburg Math. J. {\bf 14} (2003), 467–478.
	
	\bibitem{Lemenant-Milakis-Spinolo}{
	A. Lemenant, E. Milakis and L. V.  Spinolo}, {\em On the extension property of Reifenberg-flat domains}. Annales Fennici Mathematici, {\bf 39} (2014), no. 1, 51–71.
	
	\bibitem{LV}{J. Lott and C. Villani}, {\em Ricci curvature for metric-measure spaces via optimal transport}. Ann. of Math. (2), {\bf 169} (2009), no. 3, 903-991
	
	\bibitem{Luukkainen-Saksman1998}{J. Luukkainen and E. Saksman}, {\em Every complete doubling metric space 	carries a doubling measure}.
	Proc. Amer. Math. Soc. {\bf 126} (1998), no. 2,  531-534.
	
	\bibitem{Mondino-Wei2019}{A. Mondino and G. Wei}, {\em On the universal cover and the fundamental group of an RCD*(K, N)-space}. J. Reine Angew. Math. (Crelles Journal), vol. 2019, no. 753, 2019, pp. 211-237.
	
	\bibitem{Palais1961}{R. S. Palais}, {\em On the existence of slices for actions of non-compact Lie groups}. Ann. of Math. (2) {\bf 73} (1961), 295–323.
	
	\bibitem{Pan-Wang2021}{J. Pan and J. Wang}, {\em Some topological results of Ricci limit spaces}. Trans. Amer. Math. Soc. 375 (2022), 8445-8464.
	
	
	\bibitem{Pan-Wei2019}{J. Pan and G. Wei}, {\em Semi-local simple connectedness of non-collapsing Ricci limit spaces}. J. Eur. Math. Soc. 24 (2022), no. 12, 4027–4062.
	
	\bibitem{Perales2016}{R. Perales}, {\em A survey on the convergence of manifolds with boundary}. Contemp. Math. {\bf 657} (2016), 179-188.
	
	\bibitem{Perales2020}{R. Perales}, {\em Convergence of manifolds and metric spaces with boundary}.
	J. Topol. Anal. {\bf 12} (2020), no. 3, 735–774.
	
	\bibitem{Perales-Sormani2014}{R. Perales and C. Sormani}, {\em Sequences of open Riemannian manifolds with boundary}. Pacific J. Math. {\bf 270} (2014), no. 2, 423–471. 
	
	\bibitem{Perelman1997}{G. Perelman}, {\em Construction of manifolds of positive Ricci curvature with big volume and large Betti numbers}, pp. 157–163 in Comparison geometry (Berkeley, CA, 1993–94), edited by K. Grove and P. Petersen, Math. Sci. Res. Inst. Publ. 30, Cambridge Univ. Press, Cambridge, 1997.
	
	\bibitem{Petersen2016}{P. Petersen}, {\em Riemannian geometry}. Third edition. Graduate Texts in Mathematics, {\bf 171}. Springer, Cham, 2016. xviii+499 pp.
	
	\bibitem{Reifenberg1960}{E. R. Reifenberg}, {\em Solution of the  Plateau Problem for m-dimensional surfaces of varying topological type}. Acta Math. {\bf 104} (1960), 1–92.
	
	\bibitem{Sormani-Wei2001}{C. Sormani and G. Wei}, {\em Hausdorff convergence and universal covers}. Trans. Amer. Math. Soc. {\bf 353} (2001), no. 9, 3585–3602. 
	
	\bibitem{Sormani-Wei2004}{C. Sormani and G. Wei}, {\em Universal covers for Hausdorff limits of noncompact spaces}. Trans. Amer. Math. Soc. {\bf 356} (2004), no. 3, 1233–1270.
	
	\bibitem{St1}{K. Sturm}, {\em On the geometry of metric measure spaces. I}. Acta Math. {\bf 196} (2006), no. 1, 65-131.
	
	\bibitem{St2}{K. Sturm}, {\em On the geometry of metric measure spaces. II}. Acta Math. {\bf 196} (2006), no. 1, 133-177.
	
	\bibitem{Volberg-Konyagin1988}{A. L. Vol'berg and S. V. Konyagin}, {\em On measures with the doubling condition}. Math. USSR-Izv. {\bf 30} (1988), 629-638. (Russian)
	
	\bibitem{Wang2021}{J. Wang}, {\em Ricci Limit Spaces are semi-locally simply connected}. J. Differential Geom. {\bf 128} (November 2024), no. 3, 1301-1314.
	
	\bibitem{Wang2023}{J. Wang}, {\em RCD*(K,N) spaces are semi-locally simply~connected}. J. Reine Angew. Math. (Crelles Journal) {\bf 806} (2024), 1–7.
	
	\bibitem{Wong2008}{J. Wong}, {\em An extension procedure for manifolds with boundary}. Pacific J. Math. {\bf 235} (2008), 173-199.
	
	\bibitem{Xu-Yao}{S. Xu and X. Yao}, {\em Margulis lemma and Hurewicz fibration theorem on Alexandrov spaces}. Commun. Contemp. Math. {\bf 24} (2022), no. 4, 2150048.
	
	\bibitem{Yamaguchi-Zhang2019}{T. Yamaguchi and Z. Zhang}, {\em Inradius collapsed manifolds}. Geom. Topol. {\bf 23} (2019) no. 6, 2793-2860.
	
	
\end{thebibliography}
\end{document}